\documentclass{amsart}

\usepackage[all]{xy}
\usepackage[latin1]{inputenc}
\usepackage{hyperref}
\usepackage{amssymb}
\usepackage{tensor}
\newtheorem*{Maintheorem*}{Main Theorem}
\newtheorem*{theorem*}{Theorem}

\newtheorem{theorem}{Theorem}
\newtheorem{definition}{Definition}
\newtheorem{lemma}{Lemma}
\newtheorem{remark}{Remark}
\newtheorem{proposition}{Proposition}
\newtheorem{corollary}{Corollary}

\newcommand{\A}{\mathcal{A}}

\newcommand{\E}{\mathcal{E}}
\newcommand{\Oh}{\mathcal{O}}
\newcommand{\C}{\mathbb{C}}

\newcommand{\la}{\alpha}
\newcommand{\lb}{\overline{\beta}}
\newcommand{\ld}{\overline{\delta}}
\newcommand{\lc}{\gamma}
\newcommand{\lC}{\Gamma}
\newcommand{\lt}{\tau}
\newcommand{\ls}{\sigma}

\newcommand{\we}{\wedge}

\newcommand{\ovl}{\overline}
\newcommand{\dbar}{\bar \partial}
\newcommand{\dl}{ \partial}

\newcommand\<{\langle} 
\renewcommand\>{\rangle}

\newcommand{\X}{\mathcal{X}}

\newcommand{\jbar}{\ovl{\jmath}}

\newcommand{\lbar}{\ovl{l}}

\newcommand{\vbar}{\ovl{v}}

\newcommand{\Abar}{\ovl{A}}
\newcommand{\Bbar}{\ovl{B}}
\newcommand{\Cbar}{\ovl{C}}
\newcommand{\cbar}{\ovl{\gamma}}
\newcommand{\Dbar}{\ovl{D}}
\newcommand{\sbar}{\ovl{s}}

\newcommand{\id}{\operatorname{id}}

\newcommand{\laplace}{\Box_\partial}    
\newcommand{\laplacedbar}{\Box_{\dbar}}

\begin{document}

\title{Curvature of higher direct images}

\author{Philipp Naumann}
\address{Fachbereich Mathematik und Informatik,
Philipps-Universit\"at Marburg, Lahnberge, Hans-Meerwein-Straße, D-35032
Marburg,Germany}
\email{naumann@mathematik.uni-marburg.de}

\thanks{}



\date{}

\begin{abstract}
Given a holomorphic family $f:\mathcal{X} \to S$ of compact complex manifolds and a relatively ample line bundle $L\to \X$, the higher direct images $R^{n-p}f_*\Omega^p_{\X/S}(L)$ carry a natural hermitian metric. We give an explicit formula for the curvature tensor of these direct images. This generalizes a result of Schumacher \cite{Sch12}, where he computed the curvature of 
$R^{n-p}f_*\Omega^p_{\X/S}(K_{\X/S}^{\otimes m})$ for a family of canonically polarized manifolds. For $p=n$,  it coincides with a  formula of Berndtsson obtained in \cite{Be11}. Thus, when $L$ is globally ample, we reprove his result \cite{Be09} on the Nakano positivity of $f_*(K_{\X/F}\otimes L)$. 
\end{abstract}

\maketitle

\section{Introduction}
For a positive line bundle $L$ on a compact complex manifolds $X$ of dimension $n$, the cohomology groups $H^{n-p}(X,\Omega^p_X(L))$ are critical with respect to Kodaira-Nakano vanishing. More generally, we consider the higher direct image sheaves $R^{n-p}f_*\Omega^p_{\X/S}(L)$ for a proper holomorphic submersion $f: \X \to S$ of complex manifolds and a line bundle $L \to \X$ which is positive along the fibers $X_s=f^{-1}(s)$. The understanding of this situation has applications to moduli problems. In his article \cite{Sch12} Schumacher studies the case $L=K_{\X/S}$ where the fiberwise K\"ahler-Einstein metrics are used to construct a hermitian metric on the relative canonical bundle which turned out to be semi-positive on the total space. A compact curvature formula is given in this case. At first glance, the method of computation seems to be restricted to the K\"ahler-Einstein situation. In the general case, there is the result \cite{Be09} of Berndtsson about the Nakano (semi-) positivity of the direct image $f_*(K_{\X/S}\otimes L)$ in the case where $L$ is (semi-) positive. His proof relies on a careful choice of representatives of sections and the usage of his ``magic formula'' \footnote{This terminology was used in \cite{LY14}}. Relying on this method of computation, Mourougane and Takayama studied in \cite{MT08} the higher direct images $R^qf_*\Omega^n_{\X/S}(E)$ for a Nakano (semi-) positive vector bundle $E$ over $\X$. The proof given there relies on an embedding of the higher direct image into a zero'th direct image in order to apply the method of computation given in \cite{Be09}. 

In the present work we compute the curvature of the higher direct images $R^{n-p}f_*\Omega^p_{\X/S}(L)$, where $(L,h) \to \X$ is a hermitian line bundle which is positive along the fibers. The main motivation for this is the observation that Berndtsson's formula given in \cite[Th.1.2]{Be11} coincides with Schumacher's formula \cite[Th.6]{Sch12} in the case $L=K_{\X/S}$. This fact suggests that Schumacher's method of computation can be carried over to the more general setting. By putting this into practice, the magic now lies behind the technique of taking Lie derivatives of line bundle valued forms along horizontal lifts.\footnote{Very recently, the article \cite{BPW17} appeared on the arXiv. The authors prove the same curvature formula by using a different method of computation.}

\section{Differential geometric setup and statement of results}
Let $f: \X \to S$ be a proper holomorphic submersion of complex manifolds with connected fibers and $L$ a line bundle on $\X$ with hermitian metric $h$. The curvature form of the hermitian line bundle is given by
$$
\omega_{\X}:=2\pi \cdot c_1(L,h)=-\sqrt{-1}\dl\dbar\log h.
$$ 
We consider the case where the hermitian bundle $(L,h)$ is relatively positive, which means that
$$
\omega_{X_s}:=\omega_{\X}|_{X_s}
$$
are K\"ahler forms on the $n$-dimensional fibers $X_s$. Then one has the notion of the horizontal lift $v_s$ of a tangent vector $\dl_s$ on the base $S$ and we get a representative of the Kodaira-Spencer class by
$$
A_s:=\dbar(v_s)|_{X_s}.
$$
Furthermore, one sets
$$
\varphi:=\<v_s,v_s\>_{\omega_{\X}},
$$
which is called the geodesic curvature. The coherent sheaf $R^{n-p}f_*\Omega^p_{\X/S}(L)$ is locally free on $S$ outside a proper subvariety. In the case $n=p$ and $L$ ample, the sheaf $f_*(K_{\X/S}\otimes L)$ is locally free by the Ohsawa-Takegoshi extension theorem (see \cite{Be09}). We assume the local freeness of
$$
R^{n-p}f_*\Omega^p_{\X/S}(L)
$$
in the general case, hence the statement of the Grothendieck-Grauert comparison theorem holds. Now Lemma 2 of \cite{Sch12} applies, which says that we can represent local sections of $R^{n-p}f_*\Omega^p_{\X/S}(L)$ by $\dbar$-closed $(0,n-p)$-forms with values in $\Omega_{\X/S}^p(L)$ on the total space, whose restrictions to the fibers are harmonic $(p,n-p)$-forms with values in $L$. Let $\{\psi^1,\ldots,\psi^r\}$ be a local frame of the direct image consisting of such sections around a fixed point $s \in S$. We denote by $\{(\dl/\dl s_i)\;|\;i=1,\ldots,\dim S\}$ a  basis of the complex tangent space $T_sS$ of $S$ over $\C$, where $s_i$ are local holomorphic coordinates on $S$. Let $A_{i\lb}^{\la}(z,s)\dl_{\la}dz^{\lb}=\dbar(v_i)|_{X_s}$ be the $\dbar$-closed representative of the Kodaira-Spencer class of $\dl_i$ described above. Then the cup product together with contraction define maps
\begin{eqnarray*}
A_{i\lb}^{\la}\dl_{\la}dz^{\lb}\cup \; : \A^{0,n-p}(X_s,\Omega^p_{X_s}(L|_{X_s})) \to \A^{0,n-p+1}(X_s,\Omega^{p-1}_{X_s}(L|_{X_s})) \\
A_{\jbar\la}^{\lb}\dl_{\lb}dz^{\la}\cup \; : \A^{0,n-p}(X_s,\Omega^p_{X_s}(L|_{X_s})) \to \A^{0,n-p-1}(X_s,\Omega^{p+1}_{X_s}(L_{X_s}))\end{eqnarray*}  
where $p>0$ in the first and $p<n$ in the second case. Note that this is a formal analogy to the derivative of the period map in the classical case (see \cite{Gr70}). We will apply the above cup products to harmonic $(p,n-p)$-forms. In general, the results are not harmonic. 

When applying the Laplace operator to $(p,q)$-forms with values in $L$ on the fibers $X_s$, we have 
\begin{equation*}
\laplace - \laplacedbar = (n-p-q)\cdot \id
\end{equation*}
due to the definition $\omega_{X_s}=\omega_{\X}|_{X_s}$ and the Bochner-Kodaira-Nakano identity (see also the proof of Corollary \ref{InfTriv}). Thus, we write $\Box=\laplace=\laplacedbar$ in the case $q=n-p$. By considering an eigenform decomposition and using the above identity, we obtain that all eigenvalues of 
$\Box$ are $0$ or greater than $1$, hence the operator $(\Box-1)^{-1}$ exists. We use the notation 
$\psi^{\lbar}:=\ovl{\psi^l}$ for sections $\psi^l$ and write $g\,dV =\omega_{X_s}/n!$. The main result is
\begin{theorem}
\label{Thm}
\label{mainresult}
Let $f: \X \to S$ be a proper holomorphic submersion of complex manifolds and $(L,h) \to \X$ a relatively positive line bundle. With the objects described above, the curvature of $R^{n-p}f_*\Omega_{\X/S}^p(L)$ is given by 
\begin{eqnarray*}
R_{i\jbar}^{\lbar k}(s) = &&\int_{X_s}{\varphi_{i\jbar}\cdot(\psi^k\cdot\psi^{\lbar})\,g\,dV}\\
&+& \int_{X_s}{(\Box +1)^{-1}(A_i \cup \psi^k)\cdot(A_{\jbar}\cup\psi^{\lbar})\,g\,dV}\\
&+& \int_{X_s}{(\Box -1)^{-1}(A_i \cup \psi^{\lbar})\cdot(A_{\jbar}\cup\psi^{k})\,g\,dV}\\
\end{eqnarray*}
If $(L,h)\to \X$ is positive, the only contribution, which may be negative, originates from the harmonic parts in the third term
$$
-\int_{X_s}{H(A_i\cup\psi^{\lbar})\cdot H(A_j\cup\psi^{\ovl{k}})\,g\,dV}.
$$ 
\end{theorem} 
\begin{remark}
The same method of computation gives a formula for a relatively negative line bundle $(L,h)$, where we set $\omega_{\X}=-2\pi \cdot c_1(L,h)$ in this case. 
\end{remark}
\begin{corollary}[compare \cite{Be09}, Th.1.2 and \cite{Be11}, Th.1.2]
If $L\to \X$ is a (semi-) positive line bundle, which is positive along the fibers, then $f_*(K_{\X/S}\otimes L)$ is Nakano (semi)- positive.
\end{corollary}
\begin{proof}
Because of degree reasons, the third term in Theorem \ref{Thm} vanishes for $p=n$. The operator $(\Box+1)^{-1}$ is positive. Furthermore, we have
$$
\omega_{\X}^{n+1} = \omega_{\X/S}^n\sum\sqrt{-1}\varphi_{i\,\jbar}ds^i\wedge ds^{\jbar} = 
\sum\sqrt{-1}\varphi_{i\,\jbar}\cdot ds^i\wedge ds^{\jbar}\,g\,dV \quad \mathrm{modulo} \mbox{ higher order terms in } s^i, s^{\jbar}.
$$
Hence, the matrix $(\varphi_{i\jbar})$ is positive definite if $L$ is positive.
\end{proof}
\begin{corollary}[\cite{Sch12}, Th.6]
If $\X\to S$ is a family of canonically polarized compact complex manifolds, then the curvature tensor of $R^{n-p}f_*\Omega^p_{\X/S}(K_{\X/S})$ is given by
\begin{eqnarray*}
R_{i\jbar}^{\lbar k}(s) = &&\int_{X_s}{(\Box +1)^{-1}(A_i\cdot A_{\jbar})\cdot(\psi^k\cdot\psi^{\lbar})\,g\,dV}\\
&+& \int_{X_s}{(\Box +1)^{-1}(A_i \cup \psi^k)\cdot(A_{\jbar}\cup\psi^{\lbar})\,g\,dV}\\
&+& \int_{X_s}{(\Box -1)^{-1}(A_i \cup \psi^{\lbar})\cdot(A_{\jbar}\cup\psi^{k})\,g\,dV}\\
\end{eqnarray*}
\begin{proof}
The K\"ahler-Einstein metrics $\omega_{X_s}=\sqrt{-1}g_{\la\lb}(z,s)dz^{\la}\wedge dz^{\lb}$ on the fibers induce a hermitian metric on the relative canonical bundle $g^{-1}=(\det{g_{\la\lb}})^{-1}$ with curvature form $\omega_{\X}$. The K\"ahler-Einstein condition gives $\omega_{\X}|_{X_s}=\omega_{X_s}$. Furthermore, we have the elliptic equation (see \cite{Sch93})
$$
(\Box + 1)\varphi_{i\jbar} = A_i\cdot A_{\jbar}.
$$
Note also that the representatives $A_i$ are harmonic in this special case.
\end{proof}
\end{corollary}       
\begin{corollary}
\label{InfTriv}
The direct images $R^{n-p}f_*\Omega^p_{\X/S}(L)$ are all Nakano positive if $L$ is positive and $\X \to S$ everywhere infinitesimal trivial. In particular, we obtain positivity if the family $\X \to S$ is locally trivial.  
\end{corollary}
\begin{proof}
If $\X \to S$ is infinitesimal trivial, we have $A_i=\dbar(b_i)$ for a differentiable vector field $b_i^{\la}\dl_{\la}$ on the fiber 
$X_s$, because $A_i$ represents the Kodaira-Spencer class and hence needs to be $\dbar$-exact. 
The Bochner-Kodaira-Nakano identity says (on the fiber $X_s$)
$$
\laplacedbar - \laplace = \left[\sqrt{-1}\Theta(L),\Lambda\right].
$$
But by definition, we have $\omega_{X_s}=\sqrt{-1}\Theta(L)|_{X_s}$. Furthermore, it holds (see \cite[Cor.VI.5.9]{De12})
$$
\left[L_{\omega},\Lambda_{\omega}\right]u = (p+q-n)u \quad \mbox{for } u \in \A^{p,q}(X_s,L|_{X_s}).
$$
Thus, the $\laplacedbar$-harmonic $(p,n-p)$-form $\psi^{\l}$ is also harmonic with respect to $\laplace$, in particular $\dl$-exact.
Therefore, 
$$
A_i\cup\psi^{\lbar}=\dbar(b_i)\cup\psi^{\lbar}=\dbar(b_i\cup\psi^{\lbar}),
$$
so the harmonic part of $A_i\cup\psi^{\lbar}$ must vanish.
\end{proof}
Note that for a trivial fibration $X \times S \to S$, the pullback of an ample line bundle $\hat{L} \to X$ and a family of positive metrics on $L$ which give a semi-positive metric on the pullback, we obtain Nakano semi-positivity of the trivial vector bundle $H^{n-p}(X,\Omega^p_X(L))\otimes \Oh_S$ on $S$ equipped with a (possibly) non-trivial metric.   

After introducing Lie derivatives of line bundle valued forms (see also appendix \ref{LieDer}), we can use the method of computation given in \cite{Sch12,Sch13} in the more general setting. The point is that the computation given there carries over verbatim if one sets $m=1$ and replaces $K_{\X/S}$ by $L$. One has to check that there is no point where the 
$\dbar^*$-closedness of $A_s$ is used, which is a crucial fact. Moreover, there is no elliptic equation for $\varphi_{i\jbar}$ in general. Thus, we must not replace $\varphi_{i\jbar}$ by $(\Box+1)^{-1}(A_i\cdot A_{\jbar})$. Finally note that by definition
$$
\omega_{X_s} = \omega_{\X}|_{X_s}.
$$
Then the computation works without the K\"ahler-Einstein condition. We give the details of the computation in the general setting in the rest of the article.

\section{Preparations}
\subsection{Fiber integrals and Lie derivatives}
Given a family $f: \X \to S$ of compact complex manifolds $X_s$ of dimension $n$ and a $C^{\infty}$ differential form $\eta$ of degree $2n+r$,
the fiber integral
$$
\int_{\X/S}{\eta}
$$
is a differential form of degree $r$ on $S$. In particular if $\eta$ is a $(n,n)$-form, we get a function on the base $S$. If $s^1,\ldots,s^r$ are local holomorphic coordinates on the base $S$, we need to compute the derivatives
$$
\frac{\dl}{\dl s^k}\int_{X_s}\eta \quad  \mbox{for }1\leq i \leq r \quad\mbox{ and } \quad \frac{\dl}{\dl s^{\lbar}}\int_{X_s}\eta, \quad \mbox{for }
1\leq l \leq r.
$$ 
This can be done by using Lie derivatives:
\begin{lemma}[\cite{Sch12}, Lemma 1]
\label{lemma 1}
For $1\leq k \leq r$, let $w_k$ be a differentiable vector field, whose projection to $S$ equals $\dl/\dl s^k$. We write $\dl/\dl s^{\lbar}$ for 
$\dl/\dl \ovl{s^l}$ and $w_{\lbar}$ for $\ovl{w_l}$. Then
$$
\frac{\dl}{\dl s^k}\int_{X_s}\eta = \int_{X_s}L_{w_k}(\eta) \quad \mbox{and} \quad \frac{\dl}{\dl s^{\lbar}}\int_{X_s}\eta = \int_{X_s}L_{w_{\lbar}}(\eta),
$$
where $L_{w_k}$ and $L_{w_{\lbar}}$ denotes the Lie derivative in the direction of $w_k$ and $w_{\lbar}$ respectively.
\begin{proof}
By Cartan's magic formula, we have
$$
L_{w_k} = d \circ \delta_{w_k} + \delta_{w_k} \circ d,
$$
where $d$ means exterior derivative on $\X$ and $\delta_{w_k}$ contraction with the vector field $v_k$. Because $d$ commutes with the fiber integration and 
$$
\delta_{\dl/\dl s^k}\int_{\X/S}(\eta) = \int_{\X/S}\delta_{w_k}(\eta),
$$ 
the assertion follows and analogous for the second identity. 
\end{proof}
\end{lemma}

\subsection{Direct images and differential forms}
Let $f: \X \to S$ be a smooth proper family of K\"ahler manifols $X_s$ and $(\E,h)\to \X$ a hermitian holomorphic vector bundle on $\X$. We assume that the direct image $R^qf_*\E$ is locally free and furthermore that for all $s\in S$ the cohomology $H^{q+1}(X_s,\E\otimes \Oh_{X_s})$ vanishes. Thus the Grothendieck-Grauert comparison theorem holds for $R^qf_*\E$ and we can identify the fiber $R^qf_*\E\otimes_{\Oh_S}\C(s)$ with 
$H^q(X_s,\E\otimes_{\Oh_{\X}}\Oh_{X_s})$. The sections of the q-th direct image sheaf $R^qf_*\E$ can locally, after replacing $S$ by a neighborhood of a given point, be represented in terms of Dolbeault cohomology by $\dbar$-closed $(0,q)$-forms with values in $\E$. But on the fibers $H^q(X_s,\E_s)$ the K\"ahler forms and the hermitian metrics on the fibers give rise to harmonic representatives of cohomology classes. The next Lemma of Schumacher  is crucial for the later computations: 

\begin{lemma}[\cite{Sch12}, Lemma 2]
\label{lemma 2}
Let $\tilde{\Psi} \in R^qf_*\E$ be a section and $\psi_s \in \A^{0,1}(X_s,\E_s)$ the harmonic representatives of the cohomology classes 
$\tilde{\Psi}|_{X_s}$. Then locally with respect to $S$ there exists a $\dbar$-closed form $\Psi \in \A^{0,q}(\X,\E)$, which is a Dolbeault representative of $\tilde{\Psi}$ and whose restrictions to the fibers $X_s$ are $\psi_s$.
\end{lemma}
\begin{proof}
For the sake of completeness, we recall the simple argument from \cite{Sch12}. Let $\Phi \in \A^{0,q}(\X,\E)$ be a Dolbeault representative of 
$\tilde{\Psi}$. We denote by $\Phi_{\X/S}$ the induced relative $(0,q)$-form. The harmonic representatives $\psi_s$ give rise to a relative form 
$\Psi_{\X/S}$. There exists a relative $(0,q-1)$-form $\chi_{\X/S}$ on $\X$, such that the exterior derivative in fiber direction 
$\dbar_{\X/S}(\chi_{\X/S})$ satisfies
$$
\Psi_{\X/S} = \Phi_{\X/S} + \dbar_{\X/S}(\chi_{\X/S}).
$$ 
A relative form can locally be extended to a genuine form on $\X$. Denote by $\{U_i\}$ a covering of $\X$, which possesses a partition of unity $\{\rho_i\}$ such that all the restrictions $\chi_{\X/S}|_{U_i}$ can be extended to $(0,q-1)$-forms $\chi^i$ on $U_i$. Then we set 
$$
\chi := \sum \;\rho_i\,\chi^i.
$$
Because of $\chi^i_{\X/S}=\chi_{\X/S}|_{U_i}$ and the property of the partition of unity,  we have that the indued relative form of $\chi$ is indeed given by $\chi_{\X/S}$. Thus the form
$$
\Psi:=\Psi + \dbar \chi
$$ 
satisfies the requirements of the lemma. 
\end{proof}

\section{Computation of the curvature}
Computing the curvature of the $L^2$-metric on $R^{n-p}f_*\Omega_{\X/S}^p(L)$ means to take derivatives in the base direction of fiber integrals, which can be realized by taking Lie derivatives of the integrands. These Lie derivatives can be split up by introducing Lie derivatives of $(p,n-p)$-forms with values in $L$. They are computed in terms of covariant derivatives with respect to the Chern connection on 
$(X_s,\omega_s)$ and the hermitian holomorphic bundle $(L,h)|_{X_s}$.  
We use the symbol $;$ for covariant derivatives and $,$ for ordinary derivatives. Greek letters indicate the fiber direction, whereas latin indices stand for directions on the base. Because we are dealing with alternating $(p,q)$-forms, the coefficients are meant to be skew-symmetric. Thus every such $(p,q)$-form carries a factor $1/p!q!$, which we suppress in the notation. They play a role in the process of skew-symmetrizing the coefficients of a $(p,q)$-form by taking alternating sums of the (not yet skew-symmetric) coefficients. 
\subsection{Setup}
As above, we denote by $f: \X \to S$ a proper holomorphic submersion of complex manifolds, whose fibers $X_s$ have dimension $n$. We choose coordinates $z^{\la}$ on the fibers and coordinates $s^i$ on the base $S$, which together give coordinates on $\X$. We write 
$\dl_i=\dl/\dl s^i$ and $\dl_{\la}=\dl/\dl z^{\la}.$
By using the notation from above, the horizontal lift of a tangent vector $\dl_i$ is given by
$$
v_i=\dl_i + a_i^{\la}\dl_{\la}
$$
and the corresponding $\dbar$-closed representative of the Kodaira-Spencer class $\rho(\dl_i)$ by
$$
A_i:=\dbar(v_i)|_{X_s}=A_{i\lb}^{\la}(z,s)\dl_{\la}dz^{\lb}
$$
with $A_{i\lb}^{\la}=a_{i;\lb}^{\la}$. It follows immediately form the definition that these Kodaira-Spencer forms induce symmetric tensors:
\begin{corollary}
Let $A_{i\lb\,\ld}=g_{\la\lb}A_{i\ld}^{\la}$. Then
$$
A_{i\lb\,\ld}=A_{i\ld\, \lb}.
$$
\end{corollary}
By polarization, it is sufficient to treat the case where $\dim S=1$ for the computation of the curvature, which simplifies the notation. Therefore, we set $s=s^1, v_s=v_1$, etc. We write $s,\sbar$ for the indices $1,\ovl{1}$ so that 
$$
v_s = \dl_s + a_{s}^{\la}\dl_{\la}
$$
and
$$
A_s = A_{s\lb}^{\la}\dl_{\la}dz^{\lb}.
$$
We assume local freeness of the sheaf $R^{n-p}f_*\Omega_{\X/S}^p(L)$. According to Lemma \ref{lemma 2}, we can represent local sections of this sheaf by $\dbar$-closed $(0,n-p)$-forms with values in $\Omega_{\X/S}^p(L)$, which restrict to harmonic $(p,n-p)$-forms on the fibers. We denote such a section by $\psi$. In local coordinates, we have
\begin{eqnarray*}
\psi|_{X_s} &=& \psi_{\la_1 \ldots \la_p\lb_{p+1}\ldots \lb_n}dz^{\la_1}\we \ldots dz^{\la_p}\we dz^{\lb_{p+1}}\we \ldots dz^{\lb_n}\\
&=& \psi_{A_p\Bbar_{n-p}}dz^{A_p}\we dz^{\Bbar_{n-p}},
\end{eqnarray*}
where $A_p=(\la_1,\ldots,\la_p)$ and $\Bbar_{n-p}=(\lb_{p+1},\ldots,\lb_n)$. The further component of $\psi$ is 
$$
\psi_{\la_1\ldots \la_p\lb_{p+1}\ldots \lb_{n-1}\sbar}dz^{\la_1}\we \ldots \we dz^{\la_p}\we dz^{\lb_{p+1}}\we \ldots dz^{\lb_{n-1}}\we d\sbar.
$$
The $\dbar$-closedness of $\psi$ means
\begin{eqnarray}
\label{dbar-closedness}
\psi_{A_p\lb_{p+1}\ldots \lb_{n-1}\sbar;\lb_n} = \psi_{A_p\lb_{p+1}\ldots\lb_{n-1}\lb_n,\sbar}.
\end{eqnarray}
\subsection{Cup product}
\begin{definition}
Let $s \in S$ and $A=A_{s\lb}^{\la}(z,s)\dl_{\la}dz^{\lb}$ be the Kodaira-Spencer form on the fiber $X_s$. The wedge product together with the contraction define maps
\begin{eqnarray*}
A \cup \rule{0.3cm}{0.4pt} : \A^{p,n-p}(X_s,L) \to \A^{p-1,n-p+1}(X_s,L) \label{A_p}\\
\ovl{A} \cup \rule{0.3cm}{0.4pt} : \A^{p,n-p}(X_s,L) \to \A^{p+1,n-p-1}(X_s,L), \label{Abar_p}
\end{eqnarray*}
where $0 \leq p \leq n$, which can be described locally by
\begin{eqnarray*}
\left( A_{s\ld}^{\lc}\dl_{\lc}dz^{\ld}\right) \cup \left( \psi_{\la_1 \ldots \la_p \lb_{p+1}\ldots \lb_n}\;dz^{\la_1}\we \ldots \we dz^{\la_p}\we dz^{\lb_{p+1}} \we \ldots \we dz^{\lb_n}\right)\\
= A_{s\lb_p}^{\lc} \psi_{\lc\la_1\ldots \la_{p-1}\lb_{p+1}\ldots \lb_n}\; dz^{\lb_p}\we dz^{\la_1} \we \ldots \we dz^{\la_{p-1}}\we dz^{\lb_{p+1}}\we \ldots \we dz^{\lb_n},\\
\left(\Abar_{\sbar\lc}^{\ld}\dl_{\lc}dz^{\ld}\right) \cup \left( \psi_{\la_1 \ldots \la_p \lb_{p+1}\ldots \lb_n}\;dz^{\la_1}\we \ldots \we dz^{\la_p}\we dz^{\lb_{p+1}} \we \ldots \we dz^{\lb_n}\right)\\
=\Abar^{\ld}_{\sbar\la_1}\psi_{\la_2\ldots \la_{p+1}\ld\,\lb_{p+2}\ldots \lb_n}\; dz^{\la_1}\we \ldots \we dz^{\la_{p+1}}\we dz^{\lb_{p+2}}\we \ldots \we dz^{\lb_n}.
\end{eqnarray*}
\end{definition}
\subsection{Lie derivatives}
Now we choose a local frame $\{\psi^1,\ldots,\psi^r\}$ according to Lemma \ref{lemma 2}. 
The components of the metric tensor for $R^{n-p}f_*\Omega^p_{\X/S}(L)$ on the base space $S$ are given by ($q=n-p$)
$$
H^{\lbar k}(s)=\<\psi^k,\psi^l\> = \int_{X_s}{\psi^k_{A_p\Bbar_q}\psi^{\lbar}_{C_q\ovl{D}_p}g^{\ovl{D}_pA_p}g^{\Bbar_{q}C_q}h \;g\,dV},
$$
which are integrals of inner products of harmonic representatives of the cohomology classes. We also write 
$$
\psi^k \cdot \psi^{\lbar}= \psi^k_{A_p\Bbar_q}\psi^{\lbar}_{C_q\ovl{D}_p}g^{\ovl{D}_pA_p}g^{\Bbar_{q}C_q}h
$$ 
for the pointwise inner product of $L$-valued $(p,q)$-forms. When we compute derivatives with respect to the base of theses fiber integrals, we apply Lie derivatives with respect to differentiable lifts of the tangent vectors according to Lemma \ref{lemma 1}. Here we choose the \emph{horizontal} lifts, which are in particular \emph{canonical} lifts in the sense of Siu \cite{Siu86}. This simplifies the computation in a considerable way. In order to break up the Lie derivative of the pointwise inner product, we need to introduce Lie derivatives of differential forms with values in a line bundle. This can be done by using the hermitian connection $\nabla$ on $\A^{(p,q)}(X_s,L_{|X_s})$ induced by the Chern connections on $(T_{X_s},\omega_{X_s})$ and $(L,h)$. We define the Lie derivative of $\psi$ with respect to the horizontal lift $v$ by using Cartan's formula
$$
L_v\psi := (\delta_v\circ \nabla + \nabla \circ \delta_v)\psi
$$
and similar for the Lie derivative with respect to $\vbar$.
We note that this definition extends the usual Lie derivative for ordinary tensors, which can as well be computed by using covariant differentiation. We refer to the appendix for properties of Lie derivatives and a short discussion of this concept. 

Taking Lie derivatives is not type-preserving. We have the type decomposition for $\psi=\psi^k$ or $\psi=\psi^l$ and $v=v_s$
$$
L_v\psi = L_v\psi' + L_v\psi'',
$$ 
where $L_v\psi'$ is of type $(p,n-p)$ and $L_v\psi''$ is of type $(p-1,n-p+1)$. In local coordinates, we have
\begin{eqnarray}
\label{Lv'}
L_v\psi'  = \left(\psi_{A_p\Bbar_{n-p};s} + a_s^{\la}\psi_{A_p\Bbar_{n-p};\la} + \sum_{j=1}^p{a_{s;\la_j}^{\la}
\psi_{
{\tiny\vtop{
\hbox{$\la_1\ldots\la\ldots\la_p\Bbar_{n-p}$}\vskip-.8mm
\hbox{$\phantom{\la_1\ldots}{|\atop j} $}}}}}
\right)
\;dz^{A_p}\wedge dz^{\Bbar_{n-p}}
\end{eqnarray}
\begin{eqnarray}
\label{Lv''}
L_v\psi'' &=& \sum^p_{j=1} A^\la_{s\lb_p}
\psi_{ {\tiny\vtop{ \hbox{$\la_1\ldots\la\ldots\la_p\Bbar_{n-p}$\;}\vskip-.8mm \hbox{$\phantom{\la_1\ldots}{|\atop j} $}}}}
\vtop{\hbox{$dz^{\la_1}\wedge\ldots\wedge dz^{\lb_p}\wedge\ldots\wedge
dz^{\la_p} \we dz^{\lb_{p+1}}\we\ldots\we
dz^{\lb_n}$}\hbox{$\phantom{dz^{\la_1}\we\ldots\we \; }{|\atop j} $}}
\end{eqnarray}
Similarly we have a type decomposistion for the Lie derivative along $\vbar = v_{\sbar}$ 
$$
L_{\vbar}\psi = L_{\vbar}\psi' + L_{\vbar}\psi'', 
$$
where $L_{\vbar}\psi'$ is of type $(p,n-p)$ and $L_{\vbar}\psi''$ is of type $(p+1,n-p-1)$. In local coordinates, this is
\begin{eqnarray}
\label{Lvbar'}
L_{\vbar}\psi'=\left(\psi_{A_p\Bbar_{n-p};\sbar}+a_{\sbar}^{\lb}\psi_{A_p\Bbar_{n-p};\lb} + \sum_{j=p+1}^n{a_{\sbar;\lb_j}^{\lb}\psi_{
{\tiny\vtop{
\hbox{$A_p \lb_{p+1}\ldots\lb\ldots\lb_n$}\vskip-.8mm
\hbox{$\phantom{A_p \cbar_{p+1}\ldots}{|\atop j}$}}}}}
\right)
\;dz^{A_p}\wedge dz^{\Bbar_{n-p}}
\end{eqnarray}
\begin{eqnarray}
\label{Lvbar''}
L_{\vbar}\psi'' &=& \sum^n_{j=p+1} A^{\lb}_{\sbar \la_{p+1}}
\psi_{\tiny\vtop{ \hbox{$A_p\lb_{p+1}\ldots\lb\ldots\lb_n\;$}\vskip-.8mm
\hbox{$\phantom{A_p\lb_{p+1}\ldots}{|\atop j}$}}} 
\vtop{\hbox{$dz^{\la_1}\we\ldots\we dz^{\la_p}\we dz^{\lb_{p+1}}\we\ldots\we dz^{\la_{p+1}}\we\ldots
\we dz^{\lb_n} $}
\hbox{$\phantom{dz^{\la_1}\we\ldots\we dz^{\la_p}\we dz^{\la_1}\we\ldots\we dz}{|\atop j}$}}
\end{eqnarray}
We also need the following lemma
\begin{lemma}
The Lie derivative of the volume element $g\, dV=\omega_s^n/n!$ along the horizontal lift $v$ vanishes, i.e.
$$
L_v(g\, dV)=0.
$$
\end{lemma}
\begin{proof}
It suffices  to show that the $(1,1)$ component of $L_v(g_{\la\lb})$ vanishes, which implies $L_v(\det(g_{\la\lb}))=0.$ We have
$$
L_v(g_{\la\lb})_{\la\lb} = g_{\la\lb,s} + a_{s}^{\lc}g_{\la\lb;\lc} + a_{s;\la}^{\lc}g_{\lc\lb} = -a_{s\lb;\la} + a_{s;\la}^{\lc}g_{\lc\lb}=0.   
$$
\end{proof}

\subsection{Main part of the computation}
We start computing the curvature by computing the first order variation. Using Lie derivatives, the pointwise inner products can be broken up:
\begin{proposition}
\label{firstordervar}
$$
\frac{\dl}{\dl s} \<\psi^k,\psi^l\> = \<L_v\psi^k,\psi^l\> + \< \psi^k,L_{\vbar}\psi^l\>,
$$
where $\dl/\dl s$ denotes a tangent vector on the base $S$.
\end{proposition}
\begin{proof}
By Lemma \ref{lemma 1} we have ($q=n-p$)
$$
\frac{\dl}{\dl s} \<\psi^k,\psi^l\> = \int_X{L_v\left(
\psi^k_{A_p\Bbar_q}\psi^{\lbar}_{C_q\ovl{D}_p}g^{\ovl{D}_pA_p}g^{\Bbar_{q}C_q}h \;g\,dV
\right)}.
$$
The integrand is now a Lie derivative of an ordinary $(n,n)$-form. We have
$$
L_v\left(
\psi^k_{A_p\Bbar_q}\psi^{\lbar}_{C_q\ovl{D}_p}g^{\ovl{D}_pA_p}g^{\Bbar_{q}C_q}h \;g\,dV
\right) =L_v\left(
\psi^k_{A_p\Bbar_q}\psi^{\lbar}_{C_q\ovl{D}_p}g^{\ovl{D}_pA_p}g^{\Bbar_{q}C_q}h\right)g\,dV,
$$
because $L_v(g\,dV)$ vanishes. Now the Lie derivative of a function is just the ordinary derivative in the
direction of $v$, so we get (by using Einstein's summation convention and , for ordinary derivatives)
\begin{eqnarray*}
L_v\left(
\psi^k_{A_p\Bbar_q}\psi^{\lbar}_{C_q\ovl{D}_p}g^{\ovl{D}_pA_p}g^{\Bbar_{q}C_q}h\right)=
\Big(\dl_s+a_s^{\la}\dl_{\la}\Big)\left(\psi^{\lbar}_{C_q\ovl{D}_p}g^{\ovl{D}_pA_p}g^{\Bbar_{q}C_q}h\right)\\
=\left(\psi_{A_p\Bbar_q,s}+a_s^{\la}\psi_{A_p\Bbar_q,\la}\right)\left(\psi^{\lbar}_{C_q\ovl{D}_p}g^{\ovl{D}_pA_p}g^{\Bbar_{q}C_q}h\right)
+ \psi^k_{A_p\Bbar_q}\left(\psi^{\lbar}_{C_q\Dbar_p,s} + a_s^{\la}\psi^{\lbar}_{C_q\Dbar_p,\la}\right)
\left(g^{\ovl{D}_pA_p}g^{\Bbar_{q}C_q}h\right)\\
+ \psi^k_{A_p\Bbar_q}\psi^{\lbar}_{C_q\Dbar_p}\left(\dl_s(g^{\ld_1\la_1})g^{\ld_2\la_2}\ldots g^{\lb_n\lc_n}h +
g^{\ld_1\la_1}\dl_s(g^{\ld_2\la_2})\ldots g^{\lb_n\lc_n}h + \ldots + g^{\ld_1\la_1}g^{\ld_2\la_2}\ldots \dl_s(g^{\lb_n\lc_n})h
\right)\\
+ \psi^k_{A_p\Bbar_q}\psi^{\lbar}_{C_q\Dbar_p}\left(a_s^{\la}\dl_{\la}(g^{\ld_1\la_1})g^{\ld_2\la_2}\ldots g^{\lb_n\lc_n}h+
g^{\ld_1\la_1}a_s^{\la}\dl_{\la}(g^{\ld_2\la_2})\ldots g^{\lb_n\lc_n}h + \ldots + g^{\ld_1\la_1}g^{\ld_2\la_2}\ldots a_s^{\la}\dl_{\la}(g^{\lb_n\lc_n})h
\right)\\
+ \left(\psi^k_{A_p\Bbar_q}\psi^{\lbar}_{C_q\ovl{D}_p}g^{\ovl{D}_pA_p}g^{\Bbar_{q}C_q}\right) \Big(\dl_sh+a_s^{\la}\dl_{\la}h \Big)
\end{eqnarray*}  
Now we use the identities $\dl_sg^{\lb\lc}=g^{\lb\ls}a_{s;\ls}^{\lc}$ and $\dl_{\la}g^{\lb\lc}=-g^{\lb\ls}\Gamma_{\la\ls}^{\lc}$
as well as the Christoffel symbols for the Chern connection on $(L,h)$ which are $h^{-1}\dl_sh=\Gamma^h_s$
and $h^{-1}\dl_{\la}h=\Gamma^h_{\la}$. The above somewhat lengthy expression can then be written as (now we use ; for indicating covariant derivatives)
\begin{eqnarray*}
(\psi^k_{A_p\Bbar_q,s}+\Gamma^h_s\psi^k_{A_p\Bbar_q} + a_s^{\la}(\psi_{A_p\Bbar_q,\la}^k
-\sum_{j=1}^p{\Gamma^{\ls}_{\la\la_j}\psi^k_{
{\tiny\vtop{
\hbox{$\la_1\ldots\ls\ldots\la_p\Bbar_q$}\vskip-.8mm
\hbox{$\phantom{\la_1\ldots}{|\atop j} $}}}}}
+\Gamma^h_{\la}\psi^k))\psi^{\lbar}_{C_q\Dbar_p}
g^{\ovl{D}_pA_p}g^{\Bbar_{q}C_q}h\\
+ \psi^k_{A_p\Bbar_q}(\psi^{\lbar}_{C_q\Dbar_p,s}+a_s^{\la}(\psi_{C_q\Dbar_p,\la}^{\lbar}
-\sum_{j=p+1}^n{\Gamma^{\ls}_{\la\lc_j}\psi^{\lbar}_{
{\tiny\vtop{
\hbox{$\lc_{p+1}\ldots\ls\ldots\la_n\Dbar_p$}\vskip-.8mm
\hbox{$\phantom{\lc_{p+1}\ldots}{|\atop j} $}}}}}
))g^{\ovl{D}_pA_p}g^{\Bbar_{q}C_q}h\\
+(\sum_{j=1}^p{\psi^k_{
{\tiny\vtop{
\hbox{$\la_1\ldots\la\ldots\la_p\Bbar_q$}\vskip-.8mm
\hbox{$\phantom{\la_1\ldots}{|\atop j} $}}}}}
a_{s;\la_j}^{\la})
\psi^{\lbar}_{C_q\Dbar_p}g^{\ovl{D}_pA_p}g^{\Bbar_{q}C_q}h 
+ \psi^k_{A_p\Bbar_q}(\sum_{j=p+1}^n{\psi^{\lbar}_{
{\tiny\vtop{
\hbox{$\lc_{p+1}\ldots\la\ldots\lc_n\Dbar_p$}\vskip-.8mm
\hbox{$\phantom{\lc_{p+1}\ldots}{|\atop j} $}}}}}
a_{s;\lc_j}^{\la})g^{\ovl{D}_pA_p}g^{\Bbar_{q}C_q}h\\
= (\psi^k_{A_p\Bbar_q:s} + a_s^{\la}\psi^k_{A_p\Bbar_q;\la})\psi^{\lbar}_{C_q\Dbar_p}g^{\ovl{D}_pA_p}g^{\Bbar_{q}C_q}h
+ \psi^k_{A_p\Bbar_q}(\psi^{\lbar}_{C_q\Dbar_p;s}+a_s^{\la}\psi^{\lbar}_{C_q\Dbar_p;\la})g^{\ovl{D}_pA_p}g^{\Bbar_{q}C_q}h\\
+(\sum_{j=1}^p{\psi^k_{
{\tiny\vtop{
\hbox{$\la_1\ldots\la\ldots\la_p\Bbar_q$}\vskip-.8mm
\hbox{$\phantom{\la_1\ldots}{|\atop j} $}}}}}
a_{s;\la_j}^{\la})
\psi^{\lbar}_{C_q\Dbar_p}g^{\ovl{D}_pA_p}g^{\Bbar_{q}C_q}h 
+ \psi^k_{A_p\Bbar_q}(\sum_{j=p+1}^n{\psi^{\lbar}_{
{\tiny\vtop{
\hbox{$\lc_{p+1}\ldots\la\ldots\lc_n\Dbar_p$}\vskip-.8mm
\hbox{$\phantom{\lc_{p+1}\ldots}{|\atop j} $}}}}}
a_{s;\lc_j}^{\la})g^{\ovl{D}_pA_p}g^{\Bbar_{q}C_q}h
\end{eqnarray*}  
The $(p,q)$-components of the forms $L_v\psi$ and $L_{\vbar}\psi^l$ are given by
\begin{eqnarray*}
(L_v\psi^{k})_{(p,q)}  = (\psi^k_{A_p\Bbar_q;s} + a_s^{\la}\psi^k_{A_p\Bbar_q;\la} + \sum_{j=1}^p{a_{s;\la_j}^{\la}
\psi^k_{
{\tiny\vtop{
\hbox{$\la_1\ldots\la\ldots\la_p\Bbar_q$}\vskip-.8mm
\hbox{$\phantom{\la_1\ldots}{|\atop j} $}}}}}
)
dz^{A_p}\wedge dz^{\Bbar_q}
\end{eqnarray*}
and
\begin{eqnarray*}
(L_{\vbar}\psi^l)_{(p,q)}=(\psi^l_{D_p\Cbar_q;\sbar}+a_{\sbar}^{\cbar}\psi^l_{D_p\Cbar_q;\cbar} + \sum_{j=p+1}^n{a_{\sbar;\cbar_j}^{\cbar}\psi^l_{
{\tiny\vtop{
\hbox{$D_p \cbar_{p+1}\ldots\cbar\ldots\cbar_n$}\vskip-.8mm
\hbox{$\phantom{D_p \cbar_{p+1}\ldots}{|\atop j}$}}}}}
)
dz^{D_p}\wedge dz^{\Cbar_q},
\end{eqnarray*}
thus the statement of the proposition follows.
\end{proof}
The above proposition is a main reason for the use of Lie derivatives.
A second justification for using Lie derivatives is given by the following lemma, which allows us to express some components of the Lie derivatives as cup products with the Kodaira-Spencer form or the horizontal lift respectively:  
\begin{lemma}
\begin{eqnarray}
L_v\psi''&=&A_s \cup \psi,\label{id1}\\
L_{\vbar}\psi''&=&(-1)^pA_{\sbar}\cup \psi, \label{id2}\\
L_{\vbar}\psi'&=&(-1)^p\dbar(\vbar \cup \psi). \label{id3}
\end{eqnarray}
\end{lemma}
\begin{proof}
First, we prove (\ref{id1}):
We have
\begin{eqnarray*}
L_v\psi'' &=& \sum^p_{j=1} A^\la_{s\lb_p}
\psi_{ {\tiny\vtop{ \hbox{$\la_1\ldots\la\ldots\la_p\Bbar_{n-p}$\;}\vskip-.8mm \hbox{$\phantom{\la_1\ldots}{|\atop j} $}}}}
\vtop{\hbox{$dz^{\la_1}\wedge\ldots\wedge dz^{\lb_p}\wedge\ldots\wedge
dz^{\la_p} \we dz^{\lb_{p+1}}\we\ldots\we
dz^{\lb_n}$}\hbox{$\phantom{dz^{\la_1}\we\ldots\we \; }{|\atop j} $}}\\
&=& \sum^p_{j=1}
A^\la_{s\lb_p}\psi_{\la\,\la_1\ldots\la_{p-1} \Bbar_{n-p}} dz^{\lb_p}\we dz^{\la_1}\we
\ldots \we dz^{\la_{p-1}}\we dz^{\lb_{p+1}} \ldots\we dz^{\lb_n}.
\end{eqnarray*}
Similarly, we have
\begin{eqnarray*}
L_{\vbar}\psi'' &=& \sum^n_{j=p+1} A^{\lb}_{\sbar \la_{p+1}}
\psi_{\tiny\vtop{ \hbox{$\la_1\ldots\la_p\lb_{p+1}\ldots\lb\ldots\lb_n\;$}\vskip-.8mm
\hbox{$\phantom{\la_1\ldots\la_p\lb_{p+1}\ldots}{|\atop j}$}}} 
\vtop{\hbox{$dz^{\la_1}\we\ldots\we dz^{\la_p}\we dz^{\lb_{p+1}}\we\ldots\we dz^{\la_{p+1}}\we\ldots
\we dz^{\lb_n} $}
\hbox{$\phantom{dz^{\la_1}\we\ldots\we dz^{\la_p}\we dz^{\la_1}\we\ldots\we dz}{|\atop j}$}}\\
&=& \sum_{j=p+1}^n A_{\sbar \la_{p+1}}^{\lb}\psi_{\la_1\ldots\la_p\lb\lb_{p+2}\ldots\lb_n}\, dz^{\la_1}\we \ldots \we dz^{\la_p}\we dz^{\la_{p+1}}\wedge dz^{\lb_{p+2}}\wedge\ldots\we dz^{\lb_n}.
\end{eqnarray*}
Finally, we prove (\ref{id3}): We have
\begin{eqnarray*}
\vbar \cup \psi &=&\left(\psi_{A_p\sbar\,\lb_{p+1}\ldots\lb_{n-1}} + a_{\sbar}^{\lb}\psi_{A_p\lb\,\lb_{p+1}\ldots\lb_{n-1}} \right)\,dz^{A_p}\wedge dz^{\lb_{p+1}}\wedge \ldots \wedge dz^{\lb_{n-1}}\\
&=&(-1)^{n-p-1}\left(\psi_{A_p\lb_{p+1}\ldots\lb_{n-1}\sbar} + a_{\sbar}^{\lb}\psi_{A_p\lb_{p+1}\ldots\lb_{n-1}\lb} \right)\,dz^{A_p}\wedge dz^{\lb_{p+1}}\wedge \ldots \wedge dz^{\lb_{n-1}}
\end{eqnarray*}
and hence
\begin{eqnarray*}
\dbar(\vbar \cup \psi)&=& (-1)^{n-p-1}\left(
\psi_{A_p\lb_{p+1}\ldots\lb_{n-1}\sbar;\lb_n} + (a_{\sbar}^{\lb}\psi_{A_p\lb_{p+1}\ldots\lb_{n-1}\lb})_{;\lb_{n}}
\right)dz^{\lb_n}\wedge dz^{A_p}\wedge dz^{\lb_{p+1}}\wedge \ldots \wedge dz^{\lb_{n-1}}\\
&=& (-1)^p(-1)^{n-p-1}\left(
\psi_{A_p\lb_{p+1}\ldots\lb_{n-1}\sbar;\lb_n} + (a_{\sbar}^{\lb}\psi_{A_p\lb_{p+1}\ldots\lb_{n-1}\lb})_{;\lb_{n}}
\right)dz^{A_p}\wedge dz^{\lb_n}\wedge dz^{\lb_{p+1}}\wedge \ldots \wedge dz^{\lb_{n-1}}
\end{eqnarray*}
The skew-symmetrized coefficients of $(a_{\sbar}^{\lb}\psi_{A_p\lb_{p+1}\ldots\lb_{n-1}\lb})_{;\lb_{n}}dz^{A_p}\wedge dz^{\lb_n}\wedge dz^{\lb_{p+1}}\wedge \ldots \wedge dz^{\lb_{n-1}}$ are given by
\begin{eqnarray*}
\left[\dbar(a_{\sbar}^{\lb}\psi_{A_p\lb_{p+1}\ldots\lb_{n-1}\lb})\right]_{A_p\lb_{p+1}\ldots\lb_n}&=&
(-1)^p\sum_{j=p+1}^n{(-1)^{j-p-1}\left(a_{\sbar}^{\lb}\psi_{A_p\lb_{p+1}\ldots\Hat{\lb}_j\ldots\lb_n\lb}\right)_{;\lb_j}}\\
&=& (-1)^p\sum_{j=p+1}^n{(-1)^{j-p-1}\left(a_{\sbar;\lb_j}^{\lb}\psi_{A_p\lb_{p+1}\ldots\Hat{\lb}_j\ldots\lb_n\lb} + 
a_{\sbar}^{\lb}\psi_{A_p\lb_{p+1}\ldots\Hat{\lb}_j\ldots\lb_n\lb\,;\lb_j}\right)}
\end{eqnarray*}
Remember that we have
$$
L_{\vbar}\psi' = \left(\psi_{A_p\lb_{p+1}\ldots\lb_n;\sbar} + a_{\sbar}^{\lb}\psi_{A_p\lb_{p+1}\ldots\lb_n;\lb} + \sum_{j=p+1}^n{a_{\sbar;\lb_j}^{\lb}
\psi_{
{\tiny\vtop{
\hbox{$A_p \lb_{p+1}\ldots\lb\ldots\lb_n$}\vskip-.8mm
\hbox{$\phantom{A_p \lb_{p+1}\ldots}{l\atop j}$}}}}}
\right)\, dz^{A_p}\wedge dz^{\Bbar_{n-p}}.
$$
Now the identity (\ref{id3}) follows from the $\dbar$-closedness of $\psi$, that means equation (\ref{dbar-closedness})
and the $\dbar$-closedness of $\psi|_{X_s}$ along the fibers.
\end{proof}
Now we can simplify the expression for the first order variation of the metric tensor. Because of (\ref{id3}) and the harmonicity of 
$\psi^k|_{X_s}$, we have
$$
\< \psi^k,L_{\vbar}\psi^l\> = \< \psi^k,L_{\vbar}(\psi^l)' \>=\<\psi^k,(-1)^p\dbar(\vbar \cup \psi^l)\>=\<\dbar^*(\psi^k),(-1)^p(\vbar \cup \psi^l)\>=0. 
$$
for all $s \in S$. Thus by Lemma \ref{firstordervar}, we have
\begin{eqnarray*}
\frac{\dl}{\dl s} H^{\lbar k} &=& \<L_v\psi^k,\psi^l\> + \<\psi^k,L_{\vbar}\psi^l\>\\
&=& \<L_v\psi^k,\psi^l\>\\
&=& \<L_v(\psi^k)',\psi^l\>.
\end{eqnarray*}
For later computations, we need to compare Laplacians:  
\begin{lemma}
\label{BKN}
We have the following relation on the space $\A^{p,q}(X_s,L|_{X_s})$
\begin{equation}
\laplace - \laplacedbar = (n-p-q)\cdot \id
\end{equation}
In particular, the harmonic forms $\psi \in \A^{p,n-p}(X_s,L|_{X_s})$ are also harmonic with respect to $\dl$, which is the $(1,0)$- part of the hermitian connection on $\A^{p,n-p}(X_s,L|_{X_s})$.
\end{lemma} 
\begin{proof}
The Bochner-Kodaira-Nakano identity says (on the fiber $X_s$)
$$
\laplacedbar - \laplace = \left[\sqrt{-1}\Theta(L),\Lambda\right].
$$
But by definition, we have $\omega_{X_s}=\sqrt{-1}\Theta(L)|_{X_s}$. Furthermore, it holds (see \cite[Cor.VI.5.9]{De12})
$$
\left[L_{\omega},\Lambda_{\omega}\right]u = (p+q-n)\,u \quad \mbox{for } u \in \A^{p,q}(X_s,L|_{X_s}).
$$
\end{proof}
Next, we start to compute the second order derivative of $H^{\lbar k}$ and begin with
$$
\frac{\dl}{\dl s} H^{\lbar k} = \<L_v \psi^k,\psi^l\>.
$$
Now by using Proposition \ref{LieComm} from the appendix we obtain
\begin{eqnarray*}
\dl_{\sbar}\dl_s \<\psi^k,\psi^l\> &=& \<L_{\vbar}L_v\psi^k,\psi^l\> + \<L_v\psi^k,L_v\psi^l\>\\
&=& \<(L_{[\vbar,v]} + \Theta(L)_{\vbar v})\psi^k, \psi^l\> +  \<L_vL_{\vbar}\psi^k,\psi^l\> + \<L_v\psi^k,L_v\psi^l\>\\
&=& \<(L_{[\vbar,v]} + \Theta(L)_{\vbar v})\psi^k, \psi^l\> + \dl_s\<L_{\vbar}\psi^k,\psi^l\> - \<L_{\vbar}\psi^k,L_{\vbar}\psi^l\> + \<L_v\psi^k,L_v\psi^l\>.
\end{eqnarray*}
Because of $\<L_{\vbar}\psi^k,\psi^l\>\equiv 0$ for all $s\in S$ as we just saw, we get
\begin{equation}
\label{secondorder}
\dl_{\sbar}\dl_s \<\psi^k,\psi^l\> = \<(L_{[\vbar,v]} + \Theta(L)_{\vbar v})\psi^k,\psi^l\> + \<L_v\psi^k,L_v\psi^l\> - \<L_{\vbar}\psi^k,L_{\vbar}\psi^l\> .
\end{equation}
Now we treat each term on the right hand side of (\ref{secondorder}) separately. For the first summand, we have
\begin{lemma}
\begin{equation}
L_{[\vbar,v]} + \Theta(L)_{\vbar v}= [-\varphi^{;\la}\dl_{\la} + \varphi^{;\lb}\dl_{\lb},\rule{0.3cm}{0.4pt}] - \varphi \cdot \id, 
\end{equation}
where the bracket $[w,\rule{0.3cm}{0.4pt}]$ stands for a Lie derivative along the vector field $w$.
\end{lemma}
\begin{proof}
We first compute the vector field $[\vbar,v]$: 
\begin{eqnarray*}
[\vbar,v] &=& [\dl_{\sbar} + a_{\sbar}^{\lb}\dl_{\sbar},\dl_s + a_s^{\la}\dl_{\la}] \\
&=& \left(\dl_{\sbar}(a_s^{\la}) + a_{\sbar}^{\lb}a_{a|\lb}^{\la} \right)\dl_{\la} - \left(\dl_s(a_{\sbar}^{\lb}) + a_s^{\la}a_{\sbar|\la}^{\lb} \right)\dl_{\lb}
\end{eqnarray*}
Now we have
\begin{eqnarray*}
\dl_{\sbar}(a_s^{\la}) &=& -\dl_{\sbar}(g^{\lb\la}g_{s\lb}) = g^{\lb\ls}g_{\ls\sbar|\ovl{\lt}}g^{\ovl{\lt}\la}g_{s\lb} - g^{\lb\la}g_{s\lb|s}\\
&=& g^{\lb\ls}a_{\sbar\ls;\ovl{\lt}}g^{\ovl{\lt}\la}a_{s\lb} - g^{\lb\la}g_{s\sbar;\lb}
\end{eqnarray*}
Because of $\varphi = g_{s\sbar} - g_{\la \sbar}g_{s\lb}g^{\lb\la}$ the coefficient of $\dl_{\la}$ is $g^{\lb\la}\varphi_{;\lb}=\varphi^{;\la}$. In the same way the get the coefficient of $\dl_{\lb}$.
Next, we need to compute the contribution of the connection on $L$. 
Because of $\sqrt{-1}[\dl,\dbar]=\sqrt{-1}\Theta(L)=\omega_{\X}$, we have
\begin{eqnarray*}
\Theta(L)_{\vbar v} &=& -\Theta(L)_{v \vbar}\\
&=& -\left(g_{s\sbar} + a_{\sbar}^{\lb}g_{s\lb} + a_s^{\la} g_{\la \sbar} + a_{\sbar}^{\lb}a_s^{\la}g_{\la\lb}\right)\\
&=& -\varphi.
\end{eqnarray*}
\end{proof}
\begin{lemma}
\begin{equation}
\label{term1}
\<(L_{[\vbar,v]}+ \Theta(L)_{\vbar v})\psi^k,\psi^l\> = -\<\varphi \cdot\psi^k,\psi^l\> = -\int_{X_s}{\varphi \;\psi^k\cdot \psi^{\lbar}\; g\,dV}.
\end{equation}
\end{lemma}
\begin{proof}
The $\dl$-closedness of $\psi^k$ means that
$$
\psi^k_{;\la} = \sum_{j=1}^p
\psi^k_{
{\tiny\vtop{
\hbox{$\la_1 \ldots \la \ldots \la_p\Bbar_{n-p};\la_j\;$}\vskip-.8mm
\hbox{$\phantom{\la_1\ldots}{|\atop j} $}}}}.
$$
Thus
\begin{eqnarray*}
[\varphi^{;\la}\dl_{\la},\psi^k_{A_p\Bbar_{n-p}}]' &=& \varphi^{;\la}\psi^k_{;\la} + 
\sum_{j=1}^p
\varphi^{;\la}_{;\la_j}\psi^k_{
{\tiny\vtop{
\hbox{$\la_1 \ldots \la \ldots \la_p\Bbar_{n-p}\;$}\vskip-.8mm
\hbox{$\phantom{\la_1\ldots}{|\atop j} $}}}}\\
&=& \sum_{j=1}^p (
\varphi^{;\la}\psi^k_{
{\tiny\vtop{
\hbox{$\la_1 \ldots \la \ldots \la_p\Bbar_{n-p}\;$}\vskip-.8mm
\hbox{$\phantom{\la_1\ldots}{|\atop j} $}}}}
)_{;\la_j}\\
&=& \dl \left( \varphi^{;\la}\dl_{\la}\cup \psi^k\right).
\end{eqnarray*}
This leads to
\begin{eqnarray*}
\<[\varphi^{;\la}\dl_{\la},\psi^k_{A_p\Bbar_{n-p}}],\psi^l\> &=& \<[\varphi^{;\la}\dl_{\la},\psi^k_{A_p\Bbar_{n-p}}]',\psi^l\>\\
&=& \<\dl\left(\varphi^{;\la}\dl_{\la}\cup \psi^k\right),\psi^l\> = \<\varphi^{;\la}\dl_{\la}\cup \psi^k, \dl^*\psi^l\> =0.
\end{eqnarray*}
In the same way we get
$$
\<[\varphi^{;\lb}\dl_{\lb},\psi^k_{A_p\Bbar_{n-p}}],\psi^l\> =0.
$$
\end{proof}
The following proposition contains important identities that allow to obtain an intrinsic expression for the curvature: 
\begin{proposition}
\label{basicid}
\begin{eqnarray}
\dbar(L_v\psi^k)'&=&\dl(A_s\cup\psi^k), \label{eq:1}\\
\dbar^*(L_v\psi^k)'&=& 0, \label{eq:2}\\
\dl^*(A_s \cup \psi^k)&=&0, \label{eq:3}\\
(-1)^p \dbar^*(L_{\vbar}\psi^k)'&=&\dl^*(A_{\sbar}\cup\psi^k), \label{eq:4}\\
\dbar(L_{\vbar}\psi^k)'&=&0, \label{eq:5}\\
\dl(A_{\sbar}\cup\psi^k)&=&0. \label{eq:6}
\end{eqnarray}
\end{proposition}
The proof is given in the appendix. 
Now we look at the second term in (\ref{secondorder}) and decompose it into its two types:
\begin{eqnarray*}
\<L_v\psi^k,L_v\psi^l\> &=& \<(L_v\psi^k)',(L_v\psi^l)'\> - \<(L_v\psi^k)'',(L_v\psi^l)''\>\\
&=&  \<(L_v\psi^k)',(L_v\psi^l)'\> - \<A_s \cup \psi^k, A_s \cup \psi^l\>
\end{eqnarray*}
because of (\ref{id1}). At this point, one might wonder about the minus sign. The reason for this is as follows: We have
\begin{eqnarray*}
\dl_{\sbar}\dl_s\<\psi^k,\psi^l\> &=& \dl_{\sbar} \<(L_v\psi^k)',\psi^l\> \\
&=& \<L_{\vbar}(L_v\psi^k)',\psi^l\> + \<(L_v\psi^k)',(L_v\psi^l)'\>\\
&=& \<L_{\vbar}(L_v\psi^k),\psi^l\> - \<(L_{\vbar}(L_v\psi^k)'')'',\psi^l\> + \<(L_v\psi^k)',(L_v\psi^l)'\>\\
&=& \<L_{\vbar}(L_v\psi^k),\psi^l\> + \<(L_v\psi^k)',(L_v\psi^l)'\> - \<(L_v\psi^k)'',(L_v\psi^l)''\>,
\end{eqnarray*}
where for the last line we used the following lemma
\begin{lemma}
$$
\<(L_{\vbar}(L_v\psi^k)'')'',\psi^l\> = \<(L_v\psi^k)'',(L_v\psi^l)''\>
$$
\end{lemma}
\begin{proof}
Because of $(\ref{Lv''})$ we have for the pointwise inner product of $(p-1,n-p+1)$-forms
$$
(L_v\psi^k)''\cdot(L_v\psi^l)'' = A_{s\lb_p}^{\la}\psi^k_{\la\la_1\ldots \la_{p-1}\lb_{p+1}\ldots \lb_n}
A_{\sbar\delta_p}^{\ovl{\lc}}\psi^{\lbar}_{\ovl{\lc}\,\ovl{\lc}_1\ldots \ovl{\lc}_{p-1}\delta_{p+1}\ldots \delta_n}\; 
g^{\ovl{\lc}_1\la_1}\ldots g^{\ovl{\lc}_{p-1}\la_{p-1}}g^{\lb_p\delta_p} g^{\lb_{p+1}\delta_{p+1}}\ldots g^{\lb_n\delta_n}
$$
On the other hand, we get by using (\ref{Lv''}) and (\ref{Lvbar''}) for the pointwise inner product of $(p,n-p)$-forms
$$
(L_{\vbar}(L_v\psi^k)'')''\cdot \psi^l = (-1)^{p-1}A_{s\lb}^{\la}\psi^k_{\la\la_1\ldots \la_{p-1}\lb_{p+1}\ldots \lb_n}
A_{\sbar \la_p}^{\lb}\psi^{\lbar}_{\ovl{\lc}_1\ldots \ovl{\lc}_p\delta_{p+1}\ldots \delta_n}\;
g^{\ovl{\lc}_1\la_1}\ldots g^{\ovl{\lc}_p\la_p}g^{\lb_{p+1}\delta_{p+1}}\ldots g^{\lb_n\delta_n}
$$
Now we take the term $A_{\sbar \la_p}^{\lb}\psi^{\lbar}_{\ovl{\lc}_1\ldots \ovl{\lc}_p\delta_{p+1}\ldots \delta_n}\;g^{\ovl{\lc}_p\la_p}$ and rewrite it as
\begin{eqnarray*}
A_{\sbar \la_p}^{\lb}\psi^{\lbar}_{\ovl{\lc}_1\ldots \ovl{\lc}_p\delta_{p+1}\ldots \delta_n}\;g^{\ovl{\lc}_p\la_p} &=&
(-1)^{p-1}A_{\sbar \la_p}^{\lb}\psi^{\lbar}_{\ovl{\lc}_p\ovl{\lc}_1\ldots \ovl{\lc}_{p-1}\delta_{p+1}\ldots \delta_n}\;g^{\ovl{\lc}_p\la_p}\\
&=& (-1)^{p-1}A_{\sbar \la_p}^{\lb}\psi^{\lbar}_{\ovl{\lc}\,\ovl{\lc}_1\ldots \ovl{\lc}_{p-1}\delta_{p+1}\ldots \delta_n}\;g^{\ovl{\lc}\la_p}\\
&=& (-1)^{p-1}A_{\sbar}^{\lb\ovl{\lc}}\psi^{\lbar}_{\ovl{\lc}\,\ovl{\lc}_1\ldots \ovl{\lc}_{p-1}\delta_{p+1}\ldots \delta_n}\\
&=& (-1)^{p-1}A_{\sbar}^{\ovl{\lc}\lb}\psi^{\lbar}_{\ovl{\lc}\,\ovl{\lc}_1\ldots \ovl{\lc}_{p-1}\delta_{p+1}\ldots \delta_n}\\
&=& (-1)^{p-1}A_{\sbar \delta_p}^{\ovl{\lc}}\psi^{\lbar}_{\ovl{\lc}\,\ovl{\lc}_1\ldots \ovl{\lc}_{p-1}\delta_{p+1}\ldots \delta_n}\;g^{\lb\delta_p}\\
\end{eqnarray*}
Thus both expressions coincide.

\end{proof}
Now let $G_{\dl}$ and $G_{\dbar}$ the Green's operators on the spaces $\A^{p,q}(X_s,L|_{X_s})$ with respect to $\laplace$ and $\laplacedbar$ respectively. According to Lemma \ref{BKN} they coincide for $p+q=n$. Now we use normal coordinates (of the second kind) at a given point $s_0 \in S$. The condition $(\dl/\dl s)H^{\lbar k}|_{s_0}=0$ for all $k,l$ means that for $s=s_0$ the harmonic projection
$$
H((L_v\psi^k)')=0
$$     
vanishes for all $k$. Thus, using the identity $\id = H + G_{\dbar}\laplacedbar$ we can write
$$
(L_v\psi^k)'=G_{\dbar}\laplacedbar(L_v\psi^k)' = G_{\dbar}\dbar^*\dbar(L_v\psi^k)' = \dbar^*G_{\dbar}\dl(A_s\cup \psi^k)
$$
by (\ref{eq:2}) and (\ref{eq:1}). Because the form $\dbar(L_v\psi^k)'=\dl(A_s \cup \psi^k)$ is of type $(p,n-p+1)$, we have 
$G_{\dbar}=(\laplace + 1)^{-1}$ on such forms by Lemma \ref{BKN}. We proceed by
\begin{eqnarray*}
\<(L_v\psi^k)',(L_v\psi^l)'\> &=& \<\dbar^*G_{\dbar}\dl(A_s \cup \psi^k),(L_v\psi^l)'\>\\
&=&  \<G_{\dbar}\dl(A_s \cup \psi^k),\dl(A_s \cup \psi^l)\>\\
&=&  \<(\laplace + 1)^{-1}\dl(A_s \cup \psi^k),\dl(A_s \cup \psi^l)\>\\
&=&  \<\dl^*(\laplace + 1)^{-1}\dl(A_s \cup \psi^k),A_s \cup \psi^l\>.
\end{eqnarray*}
Now using (\ref{eq:3}) gives
\begin{eqnarray*}
\<(L_v\psi^k)',(L_v\psi^l)'\> &=& \<(\laplace + 1)^{-1}\laplace(A_s \cup \psi^k),A_s \cup \psi^l\>\\
&=& \<(\laplace + 1)^{-1}(\laplace + 1 - 1) (A_s \cup \psi^k),A_s \cup \psi^l\>\\
&=& \<A_s \cup \psi^k,A_s \cup \psi^l\> - \<(\laplace + 1)^{-1}(A_s \cup \psi^k),A_s \cup \psi^l\>.
\end{eqnarray*}
Altogether, we have
\begin{lemma}
\begin{equation}
\label{term2}
\<L_v\psi^k,L_v\psi^l\> = - \int_{X_s}{(\Box + 1)^{-1}(A_s \cup \psi^k)\cdot (A_{\sbar} \cup \psi^{\lbar})\, g\, dV}
\end{equation}
(We write $\Box=\laplace=\laplacedbar$ when applied to $(p-1,n-p+1)$-forms.)
\end{lemma}
Finally, we look at the third term in (\ref{secondorder}) and decompose it into its two types:
\begin{eqnarray*}
\<L_{\vbar}\psi^k,L_{\vbar}\psi^l\> &=& \<(L_{\vbar}\psi^k)',(L_{\vbar}\psi^l)'\> - \<(L_{\vbar}\psi^k)'',(L_{\vbar}\psi^l)''\>\\
&=&  \<(L_{\vbar}\psi^k)',(L_{\vbar}\psi^l)'\> - \<A_{\sbar} \cup \psi^k, A_{\sbar} \cup \psi^l\>,
\end{eqnarray*}
where we used (\ref{id2}). The identity (\ref{id3}) implies that the harmonic projections of $(L_{\vbar}\psi^k)'$ vanish. Hence we can write
\begin{eqnarray*}
\<(L_{\vbar}\psi^k)',(L_{\vbar}\psi^l)'\> &=& \<G_{\dbar}\laplacedbar (L_{\vbar}\psi^k)',(L_{\vbar}\psi^l)'\>\\
&\stackrel{(\ref{eq:5})}{=}& \<(G_{\dbar}\dbar\,\dbar^*L_{\vbar}\psi^k)',(L_{\vbar}\psi^l)'\>\\
&=& \<G_{\dbar}\dbar^*(L_{\vbar}\psi^k)',\dbar^*(L_{\vbar}\psi^l)'\>\\
&\stackrel{(\ref{eq:4})}{=}& \<G_{\dbar}\dl^*(A_{\sbar} \cup \psi^k),\dl^*(A_{\sbar} \cup \psi^l)\>.
\end{eqnarray*}
Now we have
\begin{lemma}
Let $\sum_{\nu}\rho_{\nu}$ be the eigenfunction decomposition of $A_{\sbar} \cup \psi^k$. Then all the eigenvalues 
$\lambda_{\nu} > 1$ or $\lambda_0=0$. In particular $(\Box - 1)^{-1}(A_{\sbar} \cup \psi^k)$ exists. 
\end{lemma}
\begin{proof}
We consider $\dl^*(A_{\sbar} \cup \psi^k)=\sum_{\nu}{\dl^*(\rho_{\nu})}$, which is a $(p,n-p-1)$-form, for which we have
$\laplace = \laplacedbar + \id$ according to lemma \ref{BKN}. Hence
$$
\lambda_{\nu} \dl^*(\rho_{\nu}) = \laplace \dl^*(\rho_{\nu}) = \laplacedbar \dl^*(\rho_{\nu}) + \dl^*(\rho_{\nu}).
$$
Form this equation we can read off that $\sum_{\nu}{\dl^*(\rho_{\nu})}$ is the eigenfunction decomposition of 
$\dl^*(A_{\sbar} \cup \psi^k)=\dbar^*(L_{\vbar}\psi^k)'$ with respect to $\laplacedbar$ with eigenvalues $\lambda_{\nu} - 1 \geq 0$. But this form is orthogonal to the space of $\dbar$-harmonic functions so that $\lambda_{\nu}-1=0$ does not occur. The harmonic part of $A_{\sbar} \cup \psi^k$ may be present though. 
\end{proof}
Now we can proceed as follows. The form $\dl^*(A_{\sbar} \cup \psi^k)=\dbar^*(L_{\vbar}\psi^k)'$ is orthogonal to both spaces of $\dbar$- and 
$\dl$-harmonic forms. Hence we can write
$$
G_{\dbar}\dl^*(A_{\sbar} \cup \psi^k) = (\laplace - 1)^{-1}\dl^*(A_{\sbar}\cup \psi^k),
$$
so that we have
\begin{eqnarray*}
\<(L_{\vbar} \psi^k)',(L_{\vbar} \psi^l)'\> &=& \<(\laplace - 1)^{-1}\dl^*(A_{\sbar}\cup \psi^k),\dl^*(A_{\sbar} \cup \psi^l)\>\\
&=&  \<(\laplace - 1)^{-1}\dl\dl^*(A_{\sbar}\cup \psi^k),A_{\sbar} \cup \psi^l\>\\
&=&  \<(\laplace - 1)^{-1}\laplace(A_{\sbar}\cup \psi^k),A_{\sbar} \cup \psi^l\>\\
&=& \<(\laplace - 1)^{-1}(\laplace - 1 +1) (A_{\sbar}\cup \psi^k),A_{\sbar} \cup \psi^l\>\\
&=& \<A_{\sbar} \cup \psi^k, A_{\sbar} \cup \psi^l\> + \<(\laplace - 1)^{-1}(A_{\sbar} \cup \psi^k), A_{\sbar} \cup \psi^l\>.
\end{eqnarray*}
Alltogether we get for our last term 
\begin{lemma}
\begin{eqnarray}
\label{term3}
\<L_{\vbar}\psi^k,L_{\vbar}\psi^l\> = \int_{X_s}{(\Box - 1)^{-1}(A_{\sbar} \cup \psi^k)\cdot (A_s \cup \psi^{\lbar})\; g\, dV}.
\end{eqnarray}
(We write again $\Box = \laplace = \laplacedbar$ when applied to $(p+1,n-p-1)$-forms.)
\end{lemma}
Now our main result Theorem \ref{mainresult} follows form (\ref{secondorder}), (\ref{term1}), (\ref{term2}), (\ref{term3}) and the fact that
$R^{\lbar k}_{i\jbar}(s_0)=-\dl_{\jbar}\dl_i H^{\lbar k}(s_0)$ in normal coordinates at a point $s_0 \in S$.

\appendix
\section{Proof of Proposition \ref{basicid}}
\begin{proof}[Proof of Proposition \ref{basicid}]
We start with (\ref{eq:1}): We drop the superscript $k$ and note that the tensors are skew-symmetrized as coefficients of alternating forms. We start with the identities
\begin{equation}
\psi_{;s\lb_{n+1}}=\psi_{;\lb_{n+1}s} - g_{s\lb_{n+1}}\psi = a_{s\lb_{n+1}}\psi,
\end{equation}
\begin{eqnarray*}
\psi_{;\la\lb_{n+1}} &=&  \psi_{;\lb_{n+1}\la} - g_{\la\lb_{n+1}}\psi\\
&-&\sum_{j=1}^p{\psi_{\la_1\ldots \ls\ldots\la_p\Bbar_{n-p}}}R^{\ls}_{\la_j\la\lb_{n+1}}\\
&-& \sum_{j=p+1}^n{\psi_{A_p\lb_{p+1}\ldots \ovl{\lt}\ldots \lb_n}}R^{\ovl{\lt}}_{\lb_j\la\lb_n}
\end{eqnarray*}
\begin{equation}
a^{\la}_{s;\la_j\lb_{n+1}} = A^{\la}_{s\lb_{n+1};\la_j} + a_s^{\ls}R^{\la}_{\ls\la_j\lb_{n+1}}.
\end{equation}
Starting from (\ref{Lv'}) we get using 
\begin{eqnarray*}
\dbar L_v\psi'= &&\Bigg(\psi_{;s\lb_{n+1}} + A^{\la}_{s\lb_{n+1}}\psi_{;\la} + a^{\la}_s\psi_{;\la\lb_{n+1}}\\
&+& \sum_{j=1}^p{a^{\la}_{s;\la_j\lb_{n+1}}}\psi_{\la_1\ldots \la \ldots \la_p,\Bbar_{n-p}}\\
&+& \sum_{j=1}^p{a^{\la}_{s;\la_j}\psi_{\la_1\ldots \la \ldots \la_p\Bbar_{n-p};\lb_{n+1}}} \Bigg)
dz^{\lb_{n+1}}\wedge dz^{A_p}\wedge dz^{\Bbar_{n-p}}\\
&=& \Big(A^{\la}_{s\lb_{n+1}}\psi_{;\la} + \sum_{j=1}^p{A^{\la}_{s\lb_{n+1};\la_j}\psi_{\la_1\ldots \la \ldots \la_p \Bbar_{n-p}}}
\Big) dz^{\lb_{n+1}}\wedge dz^{A_p}\wedge dz^{\Bbar_{n-p}}.
\end{eqnarray*}
Note that we also used the symmetries of the curvature $4$-tensor and the $\dbar$-closedness of $\psi$.
Because of the fiberwise $\dl$-closedness of $\psi$ this equals
\begin{eqnarray*}
&&\sum_{j=1}^p{(A^{\la}_{s\lb_{n+1}}\psi_{\la_1\ldots \la \ldots \la_p \Bbar_{n-p}})_{;\la_j}dz^{\lb_{n+1}}\wedge dz^{A_p}\wedge z^{\Bbar_{n-p}}}\\
&=& (-1)^n \sum_{j=1}^p(A^{\la}_{s\lb_{n+1}}\psi_{\la\la_2\ldots \la_p \Bbar_{n-p}})_{;\la_1}dz^{\la_1} \wedge dz^{A_{p-1}}
\wedge dz^{\lb_1} \wedge \ldots \wedge dz^{\lb_{n+1}}\\
&=& \dl\Big((-1)^n A^{\la}_{s\lb_{n+1}} \psi_{\la \la_2 \ldots \la_p \lb_{p+1}\ldots \lb_n}dz^{A_{p-1}}\wedge dz^{\Bbar_{n+1}}\Big)=\dl(A_s \cup \psi.)
\end{eqnarray*}
This proves (\ref{eq:1}).
Next, we prove (\ref{eq:2}). For this, we need the following
\begin{lemma}
$$
\dl_s(\lC^{\ls}_{\la\lc})=-a^{\ls}_{s;\la\lc}
$$
\end{lemma}
\begin{proof}
We start with
$$
\dl_sg^{\lb\ls} = -g^{\lb\la}g^{\ld\ls}\dl_sg_{\la\ld}=g^{\lb\la}a_{s;\la}^{\ls},
$$
hence
$$
g_{\la\lb}\dl_sg^{\lb\ls}=a_{s;\la}^{\ls}.
$$
This gives
\begin{eqnarray*}
a_{s;\la\lc}^{\ls}&=&g_{\la\lb}(\dl_{\lc}\dl_sg^{\lb\ls} + \lC^{\ls}_{\lc\delta}\dl_sg^{\lb\delta})\\
&=& g_{\la\lb}\dl_s\dl_{\lc}g^{\lb\ls} + (-g_{\delta\lb}\dl_{\lc}g^{\lb\ls})(-\dl_sg_{\la\lb}g^{\lb\ld})\\
&=& g_{\la\lb}\dl_s\dl_{\lc}g^{\lb\ls} + \dl_{\lc}g^{\lb\ls}\dl_sg_{\la\lb}
\end{eqnarray*}
On the other hand, we have
\begin{eqnarray*}
\dl_s(\lC^{\ls}_{\la\lc}) = \dl_s(\lC^{\ls}_{\lc\la}) = \dl_s(-g_{\la\lb}\dl_{\lc}g^{\lb\ls})=-\dl_sg_{\la\lb}\dl_{\lc}g^{\lb\ls}-g_{\la\lb}\dl_s\dl_{\lc}g^{\ld\ls}
\end{eqnarray*}
This gives the statement.
\end{proof}
Now we have
$$
\psi_{;s}=\psi_s + \lC^h_s\psi
$$
and
$$
\psi_{;\lc}= \psi_s + \lC^h_{\lc}\psi - \sum_{j=1}^p{\lC^{\ls}_{\lc\la_j}\psi_{\la_1\ldots\ls\ldots\la_p\Bbar_{n-p}}}
$$
This gives
$$
\psi_{;s\lc}=(\psi_s+\lC^h_s\psi)_{\lc} + \lC^h_{\lc}(\psi_s + \lC^h_s\psi) -  \sum_{j=1}^p{\lC^{\ls}_{\lc\la_j}(\psi_s+\lC^h_s\psi)_{\la_1\ldots\ls\ldots\la_p\Bbar_{n-p}}}
$$
and
$$
\psi_{;\lc s}=(\psi_{\lc} + \lC^h_{\lc}\psi -  \sum_{j=1}^p{\lC^{\ls}_{\lc\la_j}\psi_{\la_1\ldots\ls\ldots\la_p\Bbar_{n-p}}})_s + \lC^h_s(\psi_{\lc} + \lC^h_{\lc}\psi -  \sum_{j=1}^p{\lC^{\ls}_{\lc\la_j}\psi_{\la_1\ldots\ls\ldots\la_p\Bbar_{n-p}}}) 
$$
Hence, it follows
$$
\psi_{;s\lc} - \psi_{;\lc s} =  \sum_{j=1}^p{(\dl_s\lC^{\ls}_{\lc\la_j})\psi_{\la_1\ldots\ls\ldots\la_p\Bbar_{n-p}}}
$$
and by the preceding lemma
\begin{equation}
\label{sgamma}
\psi_{;\lc s} = \psi_{;s\lc} +  \sum_{j=1}^p{a^{\ls}_{s;\la_j\lc}\psi_{\la_1\ldots\ls\ldots\la_p\Bbar_{n-p}}}
\end{equation}
Now by differentiating the equation $g^{\lb_n\lc}\psi_{;\lc}=0$ in the direction of $s$, by using $\dl_sg^{\lb_n\lc}=g^{\lb_n\ls}a_{s;\ls}^{\lc}$ as well as (\ref{sgamma}) it follows
\begin{equation}
\label{2.1}
g^{\lb_n\lc}\psi_{;s\lc} = -\psi_{;\lc}g^{\lb_n\ls}a_{s;\ls}^{\lc} -  \sum_{j=1}^p{g^{\lb_n\lc}a^{\ls}_{s;\la_j\lc}\psi_{\la_1\ldots\ls\ldots\la_p\Bbar_{n-p}}}
\end{equation}
Next, since fiberwise $\psi$ is $\dbar^*$-closed, 
\begin{equation}
\label{2.2}
g^{\lb_n\lc}(a_s^{\la}\psi_{;\la}) = g^{\lb_n\lc}a_{s;\lc}^{\la}\psi_{;\la}
\end{equation}
and with the same argument
\begin{eqnarray}
&&g^{\lb_n\lc}
\left( 
\sum_{j=1}^p{a^{\ls}_{s;\la_j}\psi_{\la_1\ldots\ls\ldots\la_p\Bbar_{n-p}}}
\right)_{;\lc} \nonumber\\
&=&g^{\lb_n\lc} \sum_{j=1}^p{a^{\ls}_{s;\la_j\lc}\psi_{\la_1\ldots\ls\ldots\la_p\Bbar_{n-p}}}. \label{2.3}
\end{eqnarray}
Now $\dbar^*(L_v\psi')=0$ follows from (\ref{2.1}), (\ref{2.2}) and (\ref{2.3}).
We come to the $\dl^*$-closedness (\ref{eq:3}) of $A_s\cup \psi$. We need to show that
$$
\left(
A_{s\lb_{n+1}}^{\la}\psi_{\la\la_2\ldots\la_p\Bbar_{n-p}}
\right)_{;\ld}g^{\ld\la_p}
$$
vanishes. Since $\dl^*\psi=0$ fiberwise, the above equality equals
$$
A_{s\lb_{n+1};\ld}^{\la}\psi_{\la\la_2\ldots\la_p\Bbar_{n-p}}g^{\ld\la_p}.
$$
Because of the $\dbar$-closedness of $A_s$ this equals
$$
(A_s^{\la\la_p})_{;\lb_{n+1}}\psi_{\la\la_2\ldots\la_p\Bbar_{n-p}}
$$
However,
$$
A_s^{\la\la_p}=A_s^{\la_p\la}
$$
whereas $\psi$ is skew-symmetric so that also this contribution vanishes.

Next, we proof (\ref{eq:4}). We have
$$
L_{\vbar}\psi'=\psi_{;\sbar} + a_{\sbar}^{\lb}\psi_{;\lb} + \sum_{j=p+1}^n{a_{\sbar;\lb_j}^{\lb}\psi_{A_p\lb_{p+1}\ldots \lb \ldots \lb_n}}
$$
Furthermore,
$$
\psi_{;\sbar\lc}=\psi_{\lc\sbar} + g_{\lc\sbar}\psi = \psi_{;\lc\sbar} - a_{\sbar\lc}\psi,
$$
\begin{eqnarray*}
\psi_{;\lb\lc} &=& \psi_{\lc\lb} + g_{\lc\lb}\psi \\
&+& \sum_{j=1}^n{\psi_{\la_1\ldots \ls \ldots \la_p\Bbar_{n-p}}R^{\ls}_{\la_j\lc\lb}}\\
&+& \sum_{j=p+1}^n{\psi_{A_p\lb_{p+1}\ldots\ovl{\tau}\ldots\lb_n}R^{\ovl{\tau}}_{\lb_j\lc\lb}}
\end{eqnarray*}
and
$$
a_{\sbar;\lb_j\lc}^{\lb} = A_{\sbar\lc;\lb_j}^{\lb} - a_{\sbar}^{\ovl{\tau}}R^{\lb}_{\ovl{\tau}\lc\lb_j}
$$
This gives
\begin{eqnarray*}
-(-1)^p(-1)^{n-p-1}\dbar^*(L_{\vbar}\psi')&=&\left(
A_{\sbar\lc}^{\lb}\psi_{;\lb} + \sum_{j=p+1}^n{A_{\sbar\lc;\lb_j}^{\lb}\psi_{A_p\lb_{p+1}\ldots\lb\ldots\lb_n}}
\right)g^{\lb_n\lc}dz^{A_p}\wedge dz^{\lb_{p+1}}\wedge\ldots \wedge dz^{\lb_{n-1}}
\end{eqnarray*}
Now, because of the fiberwise $\dbar$-closedness of $\psi$ this equals
\begin{eqnarray*}
&&\sum_{j=p+1}^n{\big(
A_{\sbar\lc}^{\lb}\psi_{A_p\lb_{p+1}\ldots\lb\ldots\lb_n}}
\big)_{;\lb_j}g^{\lb_n\lc}dz^{A_p}\wedge dz^{\lb_{p+1}}\wedge\ldots \wedge dz^{\lb_{n-1}}\\
&=&\big(
A_{\sbar\lc}^{\lb}\psi_{A_p\lb_{p+1}\ldots\lb_{n-1}\lb}
\big)_{;\lb_n}g^{\lb_n\lc}dz^{A_p}\wedge dz^{\lb_{p+1}}\wedge\ldots \wedge dz^{\lb_{n-1}}\\
&=& -(-1)^{n-p-1}\dl^*(A_{\sbar\lc}^{\lb}\psi_{A_p\lb_{p+1}\ldots\lb_{n-1}\lb}dz^{\lc}\wedge dz^{A_p}\wedge dz^{\Bbar_{n-p-1}})\\
&=&-(-1)^{n-p-1}\dl^*(A_{\sbar}\cup\psi).
\end{eqnarray*}
Here we used that for $j \neq n$:
\begin{eqnarray*}
&&\big(A_{\sbar\lc}^{\lb}\psi_{A_p\lb_{p+1}\ldots\lb\ldots\lb_n}\big)_{;\lb_j}g^{\lb_n\lc}\\
&=& \big(A_{\sbar\lc}^{\lb}g^{\lb_n\lc}\psi_{A_p\lb_{p+1}\ldots\lb\ldots\lb_n}\big)_{;\lb_j}\\
&=& \big(A_{\sbar}^{\lb\, \lb_n}\psi_{A_p\lb_{p+1}\ldots\lb\ldots\lb_n}\big)_{;\lb_j}\\
&=& \big(A_{\sbar}^{\lb_n\lb}\psi_{A_p\lb_{p+1}\ldots\lb\ldots\lb_n}\big)_{;\lb_j}\\
&=& \big(A_{\sbar\lc}^{\lb}g^{\lb\lc}\psi_{A_p\lb_{p+1}\ldots\lb\ldots\lb_n}\big)_{;\lb_j}\\
&=& \big(A_{\sbar\lc}^{\lb_n}\psi_{A_p\lb_{p+1}\ldots\lb\ldots\lb_n}\big)_{;\lb_j}g^{\lb\lc}\\
&=& - \big(A_{\sbar\lc}^{\lb_n}\psi_{A_p\lb_{p+1}\ldots\lb_n\ldots\lb}\big)_{;\lb_j}g^{\lb\lc}\\
&=& - \big(A_{\sbar\lc}^{\lb}\psi_{A_p\lb_{p+1}\ldots\lb\ldots\lb_n}\big)_{;\lb_j}g^{\lb_n\lc},
\end{eqnarray*}
hence only the last summand in the above sum contributes.\\
Equation (\ref{eq:5}) follows from (\ref{id3}).
It remains to show (\ref{eq:6}): We have
\begin{eqnarray*}
\left(
A_{\sbar\la_{p+1}}^{\ld}\psi_{A_p\ld\,\lb_{p+2}\ldots\lb_n}
\right)_{;\lc}\\
=A_{\sbar\la_{p+1};\lc}^{\ld}\psi_{A_p\ld\,\lb_{p+2}\ldots\lb_n} + A_{\sbar\la_{p+1}}^{\ld}\psi_{A_p\ld\,\lb_{p+2}\ldots\lb_n;\lc}
\end{eqnarray*}
Now the $\dl$-closedness of $A_{\sbar}\cup\psi$ follows from the $\dl$-closedness of $A_{\sbar}$ and the $\dl$-closedness of 
$\psi$. Note that we have to consider the skew-symmetrized coefficients of $A_{\sbar}\cup\psi$. Alternatively, we can write
$$
\dl(A_{\sbar}\cup \psi) = \dl A_{\sbar} \cup \psi - A_{\sbar} \cup \dl\psi =0,
$$
because $\dl A_{\sbar}=0=\dl\psi$.
\end{proof}

\section{Lie derivatives of line bundle valued forms}
\label{LieDer}
The most important technical ingredient for the computation is the notion of a Lie derivative for $(p,q)$-forms with values in a hermitian line bundle. Thus we take the opportunity to discuss this concept in more detail. We consider more generally a real manifold $X$ together with a hermitian vector bundle $(E,h)$ over $X$. Let further $\nabla$ be a hermitian connection with respect to $h$ on $E$. We denote the contraction of a $E$-valued form with a vector field $V$ by 
$\delta_V$.
\begin{definition}
Let $V$ be a complex vector field on $X$ and $\psi \in \A^k(X,E)$. We define the Lie derivative in the direction of $V$ by the Cartan formula
$$
L_V(\psi) := (\delta_V \circ \nabla + \nabla \circ \delta_V)(\psi)
$$ 
\end{definition}
We note that this definition extends the usual Lie derivative for tensors of the form $T_{a_1\ldots a_r}^{b_1 \ldots b_s}$, which can as well be computed by using covariant differentiation on a Riemannian manifold. Because for the lack of an appropriate reference, we collect some properties for this generalized Lie derivative:
\begin{proposition}
\label{LieDerTensor}
Given a section $s \in \A^0(E)$ and a form $\la \in \A^k(X)$, we have for the Lie derivative of $\la \otimes s \in \A^k(E)$:
$$
L_v(\la \otimes s) = (L_v \la) \otimes s + \la \otimes (L_v s)
$$
\end{proposition} 
\begin{proof}
First, we have
$$
\nabla(\la \otimes s) = d (\la) \otimes s + (-1)^k \la \wedge \nabla(s)
$$
and
$$
\delta_v(\la \otimes s) = \delta_v(\la) \otimes s + (-1)^k\la \otimes \delta_v(s) = \delta_v(\la) \otimes s.
$$
This gives
\begin{eqnarray*}
L_v(\la \otimes s) &=& (\delta_v \nabla + \nabla \delta_v)(\la \otimes s)\\
&=& \delta_v(d (\la) \otimes s + (-1)^k \la \wedge \nabla(s)) + \nabla (\delta_v(\la)\otimes s)\\
&=& \delta_v(d(\la)) \otimes s + (-1)^k\delta_v(\la) \wedge \nabla(s) + (-1)^k(-1)^k\la \otimes \delta_v(\nabla(s)) + d(\delta_v(\la))\otimes s + (-1)^{k-1}\delta_v(\la)\wedge \nabla(s)\\
&=& \delta_v(d(\la)) \otimes s + \la \otimes \delta_v(\nabla(s)) + d(\delta_v(\la))\otimes s.
\end{eqnarray*}
On the other hand, we have
\begin{eqnarray*}
 (L_v \la) \otimes s + \la \otimes (L_v s) &=& [\delta_v(d(\la)) + d(\delta_v(\la))]\otimes s + \la \otimes [\delta_v(\nabla(s))+\nabla(\delta_v(s))]
\end{eqnarray*}
\end{proof}
This proposition allows us to give a local expression for the Lie derivative:
\begin{corollary}
Let 
$\psi = \psi_{\la_1\ldots \la_k}\, dx^{\la_1}\wedge \ldots \wedge dx^{\la_k}$
be the local expression for a form $\psi \in \A^k(F)$ with respect to local coordinates $x^1,\ldots, x^n$, where $(F,h)$ is a hermitian line bundle together with a hermitian connection $\nabla$. Let $v=v^i(\dl/\dl x^i)$ be a smooth vector field. Then we have in local coordinates
$$
L_v(\psi) = \left(v^{\la}\psi_{\la_1\ldots \la_k;\la} + v^{\la}_{,\la_1}\psi_{\la\la_2\ldots \la_k} + v^{\la}_{,\la_2}\psi_{\la_1\la\la_3\ldots\la_k} + \ldots +v^{\la}_{,\la_k}\psi_{\la_1\ldots \la_{k-1}\la}\right)\; dx^{\la_1} \wedge \ldots \wedge dx^{\la_k},
$$
where the symbol $;\la$ means $\nabla_{\la}$ and $,$ stands for an ordinary derivative. The ordinary derivatives can be replaced by covariant derivatives with respect to the Levi-Cevita connection if $X$ is a Riemannian manifold.  
\end{corollary}
Note that we have $L_v\,d = d\,L_v$ for the Lie derivative of an ordinary form $\la \in \A^k(X)$, which follows easily from the classical Cartan formula $L_v=\delta_vd+d \delta_v $.
\begin{proposition}
\label{LieComm}
For two vector fields $v,w$ and a $k$-form $\psi \in \A^k(F)$, we have
$$
L_v L_w (\psi) - L_w L_v (\psi) = L_{[v,w]}(\psi) + \Omega_{v w} \cdot \psi,
$$
where $\Omega$ is the curvature form of $(F,\nabla)$.
\end{proposition}
\begin{proof}
By the Ricci identity we have the expression
$$
L_v L_w (s) - L_w L_v (s) = L_{[v,w]}(s) + \Omega_{v w} \cdot s
$$
for sections $s \in \A^0(F)$. By the properties of the ordinary Lie derivative for $k$-forms $\la \in \A^k(X)$ (see for example \cite[p.140]{La95}) , we have
$$
L_v L_w (\la) - L_w L_v (\la) = L_{[v,w]}(\la).
$$
Thus for tensor products $\la \otimes s \in \A^k(F)$, we get by using Proposition \ref{LieDerTensor}
$$
L_v(\la \otimes s) = (L_v\la) \otimes s + \la \otimes (L_v s),
$$
that
\begin{eqnarray*}
L_w(L_v(\la \otimes s)) = L_w(L_v(\la)) \otimes s + L_v(\la)\otimes L_w(s) + L_w(\la) \otimes L_v(s) + \la \otimes L_w(L_v(s))
\end{eqnarray*}
and analogously
\begin{eqnarray*}
L_v(L_w(\la \otimes s)) = L_v(L_w(\la)) \otimes s + L_w(\la)\otimes L_v(s) + L_v(\la) \otimes L_w(s) + \la \otimes L_v(L_w(s)).
\end{eqnarray*}
Hence we get
\begin{eqnarray*}
(L_vL_w - L_wL_v)(\la \otimes s) &=& (L_{[v,w]}\la) \otimes s + (L_{[v,w]} - \la \otimes\Omega_{v w})(s)\\
&=& L_{[v,w]}(\la \otimes s) - \Omega_{v w}\cdot \la \otimes s.
\end{eqnarray*}
\end{proof}
Given two sections $\varphi \in \A^{k}(E)$ and $\psi \in A^l(E)$ of the form
$$
\varphi = \la \otimes s \mbox{ and } \psi = \beta \otimes t
$$
for $\la\in A^k, \beta \in A^l$ and differentiable sections $s,t \in \Gamma(X,E)$, we can define a pointwise inner product by
$$
h(\varphi,\psi):= (\la \wedge \beta) \cdot h(s,t),
$$
which is an element of $A^{k+l}$.
The main point for the computation is the following 
\begin{proposition}
\label{breakup}
$L_V(h(\varphi,\psi)) = h(L_V(\varphi),\psi) + h(\varphi,L_V(\psi)).$
\end{proposition} 
\begin{proof}
By writing $\varphi = \la \otimes s$ and $\psi=\beta \otimes t$ as well as $\nabla(s)=\la'\otimes s'$ and $\nabla(t)=\beta'\otimes t'$ for $1$-forms $\la',\beta'$ and sections $s',t'$ of $E$, we can check easily
$$
d(h(\varphi,\psi)) = h(\nabla(\varphi),\psi) + (-1)^k h(\varphi,\nabla(\psi))
$$
and
$$
\delta_V(h(\varphi,\psi)) = h(\delta_V(\varphi),\psi) + (-1)^kh(\varphi,\delta_V(\psi)).
$$
Because $h(\varphi,\psi)$ is a genuine form on $X$, we have
$L_V(h(\varphi,\psi))=(\delta_V \circ d + d + \delta_V)h(\varphi,\psi)$. 
We compute
\begin{eqnarray*}
(\delta_V  \circ d)h(\varphi,\psi) &=& \delta_V(h(\nabla(\varphi),\psi) + (-1)^k h(\varphi,\nabla(\psi)))\\
&=& h(\delta_V\nabla\varphi,\psi) + (-1)^{k+1}h(\nabla \varphi,\delta_V\psi) + (-1)^kh(\delta_V\varphi,\nabla \psi) + (-1)^k(-1)^kh(\varphi,\delta_V\nabla \psi).
\end{eqnarray*}
and
\begin{eqnarray*}
(d \circ \delta_V)h(\varphi,\psi) &=& d (h(\delta_V(\varphi),\psi) + (-1)^k h(\varphi,\delta_V(\psi)))\\
&=& h(\nabla\delta_V\varphi,\psi) + (-1)^{k-1}h(\delta_V \varphi,\nabla\psi) + (-1)^kh(\nabla\varphi,\delta_V \psi) + (-1)^k(-1)^kh(\varphi,\nabla \delta_V \psi).
\end{eqnarray*}
The summation of both expressions leads to a cancelation of four summands. The remaining sum is
$h(L_V(\varphi),\psi) + h(\varphi,L_V(\psi))$ as required.
\end{proof}
Because we apply Lie derivatives for hermitian line bundle valued forms on compact complex K\"ahler manifolds, the statements in this section need to be adapted to the complex case.
For convincing the reader of this point of view, we start a short discussion of this new concept.
\begin{remark}
If $\varphi=s$ and $\psi=t$ are just sections of a vector bundle, the Lie derivative is nothing else but the covariant derivative. This is used in \cite{TW04} in their computation of the curvature. The preceding proposition means that the Lie derivative of the metric $h$ vanishes, which just means that the covariant derivative vanishes.
\end{remark}
\begin{remark}
We now discuss a result obtained in \cite{Siu86} and \cite{Sch93} resp., which says that in the context of K\"ahler-Einstein metrics of constant Ricci curvature described (for the negative case) in the introduction, the Lie derivative of the K\"ahler-Einstein metric in the direction of the horizontal lift vanishes:
$$
L_{v_i}(\omega_{X_s}^n)=0.
$$
Following \cite{Sch93}, the proof is briefly written down as:
$$
\left(L_{v_i}(g_{\la\lb})\right)_{\la\lb}=\left[ \dl_i + a_i^{\lc}\dl_{\lc},g_{\la\lb}\right]_{\la\lb} = \dl_i(g_{\la \lb}) + a_i^{\lc}g_{\la\lb;\lc} + a_{i;\la}^{\ls}g_{\ls\lb} = g_{i\lb;\la} + a_{i\lb;\la}=0.
$$
Here we used the ordinary definition of Lie derivatives. The expression $ \dl_i(g_{\la \lb}) + a_{i;\la}^{\ls}g_{\ls\lb}$ can be read as 
a covariant derivative $g_{\la\lb;i}$, because 
$$
a_{i;\la}^{\ls}=-(g^{\lb\ls} g_{i\lb})_{;\la}=-g^{\lb\ls}g_{i\lb,\la}=-g^{\lb\ls}g_{\la\lb,i}=-\Gamma_{i\la}^{\ls}
$$ 
is the Christoffel symbol for the connection on $(K_{\X/S},g^{-1})$. This interpretation agrees with our extended concept of a Lie derivative which says that $L_{v_i}(g)=\nabla_{v_i}(g)=0$, where we read $g=\det(g_{\la\lb})$ as a hermitian metric on the abstract line bundle $K_{\X/S}^{-1}$ (forget about the indices). 
\end{remark}
\begin{remark}
In the computation of the curvature of the Weil-Petersson metric for a family of K\"ahler-Einstein manifolds, one needs to compute the Lie derivative of
$$
A_i\cdot A_{\jbar}=A_{i\lb}^{\la}A_{\jbar\lc}^{\ld}g_{\la\ld}g^{\lb\lc}
$$
By taking classical Lie derivatives, this equals
$$
(L_{v_k}A_i)^{\la}_{\lb}=\dl_k(A_{i\lb}^{\la})+a_k^{\ls}A_{i\lb;\ls}^{\la} - a_{k;\ls}^{\la}A_{i\lb}^{\ls}
$$ 
If we view the elements $A_{i\lb}^{\la}\dl_{\la}dz^{\lb}$ as $(0,1)$-forms with values in the (abstract) hermitian vector bundle 
$(T_{X_s},(g_{\la\lb}))$, we obtain for the Lie derivative 
$$
(L_{v_k}A_i)^{\la}_{\lb}=(\nabla_{v_k}A_i)_{\lb}^{\la}=A_{i\lb;k}^{\la} + a_k^{\ls}A_{i\lb;\ls}^{\la}=A_{i\lb,k}^{\la}+\Gamma_{k\ls}^{\la}A_{i\lb}^{\ls} + a_k^{\ls}A_{i\lb;\ls}^{\la}
$$
where $\Gamma_{k\ls}^{\la}=g^{\ld\la}\dl_kg_{\ls\ld}=-a_{k;\ls}^{\la}$ is the Christoffel symbol for the Chern connection on the hermitian vector bundle $(T_{X_s},(g_{\la\lb}))$. Indeed, both expressions coincide.   
\end{remark}

\end{document}